\documentclass[fleqn,11pt,oneside]{article}

\usepackage{soul}
\usepackage{hyperref}
\usepackage[final]{graphicx}
\usepackage{bbm}  %
\usepackage{bm}  %
\usepackage{amsmath,amsthm,amssymb,amscd}
\usepackage{caption,subcaption}
\usepackage{booktabs,multirow}
\usepackage{setspace} 
\usepackage[top=1.1in, bottom=1.1in, left=1.6in, right=1.1in]{geometry}
\usepackage[final]{showlabels}
\usepackage{dashbox}
\usepackage{microtype}
\usepackage[medium]{titlesec}
\usepackage{srcltx}

\newtheorem{theorem}{Theorem}[section]

\newtheorem{definition1}{Definition}[section]
\newtheorem{observe}{Observation}[section]
\newtheorem{remark1}[observe]{Remark}
\newtheorem{example1}{Example}[section]
\newtheorem{aside1}[observe]{Aside}

\newenvironment{definition}[1][]{\begin{definition1}[#1] \rm}{\end{definition1}}

\newenvironment{remark}{\begin{remark1} \rm}{\end{remark1}}

\def\qed{\hfill$\blacksquare$\\}

\newif\ifshowboxes \showboxestrue

\providecommand{\e}[1]{\ensuremath{\times 10^{#1}}}

\renewcommand\Re{\operatorname{Re}}

\renewcommand{\d}{\,\mathrm{d}}

\newcommand{\pp}[2]{\frac{\partial #1}{\partial #2}}

\newcommand{\norm}[1]{\ensuremath{ {\lVert #1 \rVert} }}

\newcommand{\abs}[1]{\ensuremath{ {\lvert #1 \rvert} }}

\def\C{\mathbbm{C}}
\def\R{\mathbbm{R}}

\def\1{\mathbbm{1}}
\def\ft{\mathcal{F}}
\def\nt{\mathcal{N}}
\def\st{\mathcal{S}}

\def\O{\mathcal{O}}

\usepackage{color} %
\usepackage{xcolor}
\usepackage{xspace}
\renewcommand{\tilde}{\widetilde}
\renewcommand{\hat}{\widehat}
\newcommand{\tDelta}{\widetilde\Delta}
\newcommand{\bDelta}{\Delta_1}
\newcommand{\td}[1][]{\textcolor{red}{\fcolorbox{red}{pink}{\textbf{\textcolor{red}{\scalebox{0.7}[1.0]{\small TODO}}}}\ifthenelse{\equal{#1}{}}{}{~\emph{#1}}}\xspace}

\newcommand{\tri}{\Delta^0}
\newcommand{\origin}{O}

\newcommand{\mach}{u}

\begin{document}

\begin{center}
    \begin{minipage}[t]{6.0in}
      The accurate and efficient evaluation of Newtonian potentials
      over general 2-D domains is  important for the numerical
      solution of Poisson's equation and volume integral equations. 
      In this paper, we present a simple and efficient high-order algorithm
      for computing the Newtonian potential over a planar domain discretized
      by an unstructured mesh.  The algorithm is based on the use of Green's
      third identity for transforming the Newtonian potential into a
      collection of layer potentials over the boundaries of the mesh
      elements, which can be easily evaluated by the Helsing-Ojala method.
      One important component of our algorithm is the use of high-order (up
      to order 20) bivariate polynomial interpolation in the monomial basis,
      for which we provide extensive justification. The performance of our
      algorithm is illustrated through several numerical experiments.

 \vspace{ 0.15in}

   \vspace{ -100.0in}
 
 \thispagestyle{empty}

   \end{minipage}
 \end{center}
 
 \vspace{ 3.10in}
 \vspace{ 0.80in}
 
 \begin{center}
   \begin{minipage}[t]{4.4in}
     \begin{center}

\textbf{Rapid evaluation of Newtonian potentials on planar~domains}
\\
   \vspace{ 0.30in}
 
 Zewen Shen$\mbox{}^{\dagger\, \diamond\, \star}$ and
 Kirill Serkh$\mbox{}^{\ddagger\, \diamond}$  \\
  v4, Oct 3, 2023  
 
     \end{center}
   \vspace{ -50.0in}
   \end{minipage}
 \end{center}
 
 \vspace{ 1.05in}

 \vfill
 
 \noindent 
 $\mbox{}^{\diamond}$  This author's work was supported in part by the NSERC
 Discovery Grants RGPIN-2020-06022 and DGECR-2020-00356.
 \\

 \vspace{2mm}
 
 \noindent
 $\mbox{}^{\dagger}$ Dept.~of Computer Science, University of Toronto,
 Toronto, ON M5S 2E4\\
 \noindent
 $\mbox{}^{\ddagger}$ Dept.~of Math. and Computer Science, University of Toronto,
 Toronto, ON M5S 2E4 \\
 
 \vspace{2mm}
 \noindent 
 $\mbox{}^{\star}$  Corresponding author
 \\

 \vfill
 \eject
\tableofcontents

\section{Introduction}

The accurate and efficient discretization of the Newtonian potential integral
operator
  \begin{align}
\mathcal{N}_\Omega[f](x):=\frac{1}{2\pi}\iint_{\Omega} \log(\norm{x-y})f(y)\d A_y
\label{for:intro}
  \end{align}
for a complicated 2-D domain $\Omega$ is important for the
numerical solution of Poisson's equation and volume integral equations.
However, its numerical evaluation poses three main difficulties.
Firstly, the integrand is weakly-singular, and thus, special-purpose
quadrature rules are required. Secondly, a complicated
domain $\Omega$ typically requires at least part of the domain to be
discretized by an unstructured mesh, over which the direct evaluation of the
potential by quadrature becomes costly. Finally, the
algorithm for evaluation should have linear time complexity with small
constants.  

When solving Poisson's equation, the Newtonian potential is used as a
particular solution to the equation. This particular solution can be
obtained by evaluating the volume integral in (\ref{for:intro}) directly
(see, for example, \cite{volint,anderson,zhuthesis,fata}), or,
alternatively, can be obtained by computing the Newtonian potential over a
regular domain $\Omega^+ \supset \Omega$ for an extended density function
$f^+$ defined on $\Omega^+$, such that $f^+|_\Omega=f|_\Omega$, which allows
for efficient precomputations for accelerating the potential evaluation
\cite{box1,askham}. When following the latter approach, the order of
convergence depends on the smoothness of the extended density function
$f^+$, which means that $f^+$ must be sufficiently smooth over $\Omega^+$ in
order to reach high accuracy within a reasonable computational budget.  We
refer the readers to, for example, \cite{fry1,askham,jiang1,bruno}, for a
series of work along this line. 

When solving volume integral equations, the aforementioned function extension
method is no longer applicable, as the computation does not require a
particular solution to Poisson's equation, but rather, a discretization
of the operator~(\ref{for:intro}).  However, as is shown in
\cite{volint,anderson}, difficulties arise when the domain is discretized by
an unstructured mesh, and a quadrature-based method is used. Firstly, the
Newtonian potential generated over a mesh element at a target location close
to that element is costly to compute, as the integrand is nearly-singular,
and thus, expensive adaptive integration is generally required. Furthermore,
one cannot efficiently precompute these near interactions as is done in
\cite{box1}, since the relative position of the target location and the
nearby mesh elements is arbitrary when an unstructured mesh is used.
Secondly, efficient self-interaction computations (i.e., when the target
location is inside the mesh element generating the Newtonian potential)
generally require a large number of precomputed generalized Gaussian
quadrature rules \cite{bremer1,ggq}, which could be nontrivial to construct.

There are several previously proposed methods \cite{mayo1,dual,ethridge,tower}
which avoid these issues, by not directly evaluating the volume integral.
One such method is the dual reciprocity method (DRM) \cite{dual}, which first
constructs a global approximation of the anti-Laplacian of the density
function over the domain, and then reduces the evaluation of the Newtonian
potential over the domain to the evaluation of layer potentials over the
boundary of the domain by Green's third identity. As the 1-D layer potential
evaluation problem has been studied extensively, such a reduction is
favorable. Furthermore, the method does not require the domain to be meshed,
and thus, is particularly suitable for use in the boundary integral
equation method \cite{fds}. However, approximating the density function
globally over the domain using, for example, radial basis
functions with tractable anti-Laplacians, is challenging, and the method is
often inefficient when high accuracy is required. 

In this paper, we present a simple and efficient high-order algorithm that unifies the
far, near and self-interaction computations, and resolves all of the
aforementioned problems. As in the DRM, we use the anti-Laplacian to reduce
the volume integral to a collection of boundary integrals. However, unlike
the DRM, we approximate the anti-Laplacian locally over each mesh element,
and then reduce the Newtonian potential to layer potentials over the
boundaries of the individual mesh elements. We efficiently evaluate the
resulting layer potentials to machine precision using the Helsing-Ojala
method. As a result, we are able to rapidly evaluate the Newtonian potential
generated by each mesh element at any target location to machine accuracy,
with the speed of the evaluation independent of the target location.  In
particular, the speeds of close and self-evaluations for a single mesh
element are almost the same as the speed of evaluating a layer potential
over the element boundary by naive quadrature. Furthermore, the use of
Green's third identity reduces the number of quadrature nodes in the far
field interaction computation over a single mesh element from $\O(N^2)$ to
$\O(N)$.  Finally, we note that the precomputation required by our algorithm
makes up a small fraction of the total cost.

The key component of our algorithm is the computation of the anti-Laplacian
of the density function $f$ over each mesh element. We approximate $f$ by a
bivariate polynomial interpolant in the monomial basis, which allows for
easy computation of the anti-Laplacian using simple recurrence relations,
and provides a unified approach for handling both triangle and curved
triangle mesh elements.  Despite the exponential ill-conditioning of the
Vandermonde matrix, we recently show in \cite{mono} that the 
monomial basis generally performs as well as a well-conditioned polynomial
basis for interpolation, provided that the condition number of the
Vandermonde matrix is below the reciprocity of machine epsilon. In this
paper, we apply this idea to bivariate polynomial interpolation in the
monomial basis over a (possibly curved) triangle, and demonstrate that the
resulting order of approximation can reach up to 20, regardless of the
triangle's aspect ratio.

One may observe that our algorithm resembles the method proposed in Chapter
5 of \cite{ethridge}.  However, there exist two notable distinctions.
Firstly, the order of approximation is constrained to 4 in \cite{ethridge},
whereas our approach permits a substantially higher order of approximation,
reaching up to 20.  Secondly, we discretize the domain solely by (possibly
curved) triangles, while in \cite{ethridge}, the domain is discretized by
the Cartesian cut cell method, where the potentials generated over the
interior boxes are computed by the box code \cite{box1}, and the ones
generated over the cut cells are computed via Green's third identity.

\section{Mathematical and numerical preliminaries}

\subsection{Newtonian potential}
  \label{sec:newton}
\begin{definition}
The infinite-space Green's function for Poisson's equation is 
  \begin{align}
G(x,y)=\frac{1}{2\pi}\log\norm{x-y},
  \end{align}
where $x,y\in\R^2$.
\end{definition}
It is well-known that the function $G$ satisfies
  \begin{align}
\nabla^2_x G(x,y) = \delta(x-y), \label{for:Gdelta}
  \end{align}
where $\delta$ denotes the Dirac delta function. 

\begin{definition}
Given a domain $\Omega$ and an integrable function $f:\Omega\to\R$,
the Newtonian potential with density $f$ is defined to be
  \begin{align}
u(x)=\iint_{\Omega} G(x,y)f(y)\d A_y = \frac{1}{2\pi}\iint_{\Omega}
\log(\norm{x-y})f(y)\d A_y. \label{for:newpot}
  \end{align}
\end{definition}

It follows immediately from (\ref{for:Gdelta}) that the Newtonian potential
$u(x)$ satisfies $\nabla^2 u=f$ in~$\Omega$.

We now introduce Green's third identity, which reduces the Newtonian
potential over~$\Omega$ to layer potentials over $\partial\Omega$.  
\begin{theorem}
\label{thm:green3}
Let $\Omega$ be a 2-D planar domain and $f$ be an integrable function on
$\Omega$. Suppose that $\varphi:\Omega\to\R$ satisfies $\nabla^2 \varphi=f$. Then,
  \begin{align}
\hspace*{-5em}\iint_\Omega G(x,y)f(y)\d A_y =
\varphi(x)\mathbbm{1}_\Omega(x)+\oint_{\partial \Omega} \Bigl(
G(x,y)\pp{\varphi}{n_y}(y) - \pp{G(x,y)}{n_y} \varphi(y) \Bigr)\d \ell_y,
  \label{for:green3}
  \end{align}
for $x\in \R^2\setminus \partial\Omega$, where $\mathbbm{1}_\Omega$ denotes
the indicator function for the domain $\Omega$, and $n_y$ denotes the
outward pointing unit normal vector at the point $y$.  
\end{theorem}

\subsection{The Helsing-Ojala method for the close evaluation of 1-D
layer potentials}
  \label{sec:helsing}

In this section, we review the Helsing-Ojala method \cite{helsing} for
accurate and efficient evaluation of the 1-D single- and double-layer
potentials 
  \begin{align}
\int_{\Gamma} G(x,y)\pp{\varphi}{n_y}(y) \d \ell_y \text{\quad and \quad}
\int_{\Gamma} \pp{G(x,y)}{n_y} \varphi(y) \d \ell_y,
\label{for:sdpot}
  \end{align}
where $x\in\R^2$ is in close proximity to the curve $\Gamma\subset\R^2$.
Without loss of generality, we assume that the left endpoint of $\Gamma$
is $(-1,0)$, and the right endpoint of $\Gamma$ is $(1,0)$.

Firstly, observe that
  \begin{align}
\int_{\Gamma} G(x,y)\pp{\varphi}{n_y}(y) \d
\ell_y=\frac{1}{2\pi}\Re\int_{\Gamma}\log(z-x)
\Bigl(\pp{\varphi}{n_y}(z)\cdot
\frac{\d\ell_y}{\d z}\Bigr) \d z, \label{for:spot2}
  \end{align}
and 
  \begin{align}
\int_{\Gamma} \pp{G(x,y)}{n_y} \varphi(y) \d
\ell_y=\Re\frac{1}{2\pi i}\int_{\Gamma} \frac{\varphi(z)}{z-x}\d z,\label{for:dpot2}
  \end{align}
where, in a slight abuse of notation, we equate $\R^2$ with $\C$.
The integrals $\int_{\Gamma} \frac{z^{k}}{z-x}\d z$ and
$\int_{\Gamma}\log(z-x) z^k \d z$ satisfy the following recurrence
relations:
  \begin{subequations}
  \begin{align}
\hspace*{-4em}\int_{\Gamma} \frac{1}{z-x}\d z
=&\,\log(1-x)-\log(-1-x)+2\pi i \mathcal{N}_{x},\\
\hspace*{-4em}
\int_{\Gamma} \frac{z^{k+1}}{z-x}\d z
=&\,x\int_{\Gamma} \frac{z^k}{z-x}\d z
+\frac{1+(-1)^k}{k+1},\\
\hspace*{-4em}
\int_{\Gamma}\log(z-x) z^k \d
z=&\,\frac{1}{k+1}\Bigl(\log(1-x)+(-1)^k\log(-1-x)-\int_{\Gamma}
\frac{z^{k+1}}{z-x}\d z\Bigr),
  \end{align}
  \end{subequations}
for all $k\geq 0$, where $\mathcal{N}_x = 0$ when $x$ is outside the
region enclosed by the oriented closed curve formed by $\Gamma$ (traversed
forwards) and $[-1,1]$ (traversed backwards), and $\mathcal{N}_x = +1\,
(-1)$ when $x$ is inside the region enclosed counterclockwise (clockwise).
We note that these recurrence relations are stable when $x$ is close
to $\Gamma$.  Therefore, if the complex density functions
$\pp{\varphi}{n_y}(z)\cdot \frac{\d\ell_y}{\d z}$ and $\varphi(z)$ in
\eqref{for:spot2} and \eqref{for:dpot2} are approximated uniformly to high
accuracy by complex polynomials expressed in the monomial basis, then the
single- and double-layer potentials (\ref{for:sdpot}) can be readily
calculated via (\ref{for:spot2}) and (\ref{for:dpot2}), respectively, with
the aid of the aforementioned recurrence relations.

To approximate the density functions $\pp{\varphi}{n_y}(z)\cdot
\frac{\d\ell_y}{\d z}$ and $\varphi(z)$ by complex polynomials in the
monomial basis, one collocates at a set of nodes over $\Gamma$ with a small
Lebesgue constant, and then solves the resulting Vandermonde system with a
backward stable solver. Despite the ill-conditioning of the Vandermonde
matrix, based on our analysis in \cite{mono}, the monomial basis is as good
as a well-conditioned polynomial basis for interpolation, provided that the
condition number of the Vandermonde matrix is smaller than $\frac{1}{u}$,
where $u$ denotes machine epsilon, and that $u\cdot \norm{a}_2$ is smaller
than the polynomial interpolation error, where $a$ denotes the monomial
coefficient vector of the interpolating polynomial.  As is shown in
\cite{mono}, the first condition is met when the order of
approximation is less than $\approx 40$, even in the case where $\Gamma$ 
has a high curvature. In addition, the second condition is satisfied
automatically in most practical situations. Therefore, the use of a monomial
basis in floating point arithmetics is justified under these conditions.
However, it is pointed out in \cite{barnett} that the complex
density functions $\pp{\varphi}{n_y}(z)\cdot \frac{\d\ell_y}{\d z}$ and
$\varphi(z)$ have a singularity close to the domain $\Gamma$ when the
curvature of $\Gamma$ is not small, which leads to a slowly decaying
polynomial interpolation error. This issue can be remedied by adaptively
subdividing $\Gamma$ until the curvature of each subpanel is small.

\section{Bivariate polynomial interpolation in the monomial basis}
  \label{sec:mono}
In this section, we discuss the numerical stability of bivariate polynomial
interpolation in the monomial basis over a (possibly curved) triangle.

Let $\Delta\subset \R^2$ be a triangle, and let $F:\Delta\to\R$ be an
arbitrary function.  We define $\tilde N$ to be the dimensionality of the
space of 2-D polynomials of degree at most $N$, which is equal to
$\frac{(N+1)(N+2)}{2}$. The $N$th degree interpolating polynomial, which we
denote by $P_N$, of the function $F$ for a given set of $\tilde N$
collocation points $Z:=\{(x_j,y_j)\}_{j=1,\dots,\tilde N}\subset \Delta$ can
be expressed as 
  \begin{align}
P_N(x,y):=\sum_{j=0}^N \sum_{k=0}^{j} a_{j-k,k}
\Bigl(\frac{x-c}{s}\Bigr)^{j-k}\Bigl(\frac{y-d}{t}\Bigr)^k,
  \end{align}
where $(c,d)\in\R^2$ is the monomial expansion center, $s,t\in\R$ are the
scaling factors of the basis, and the monomial coefficient vector
$a^{(N)}:=(a_{00},a_{10},a_{01},\dots,a_{0N})^T\in \R^{\tilde N}$ is the
solution to the Vandermonde system $V^{(N)}a^{(N)}=f^{(N)}$, where
  \begin{align}
  \hspace*{-5.2em}
V^{(N)}:=  
\begin{pmatrix}
1 & \frac{x_1-c}{s} & \frac{y_1-d}{t} & (\frac{x_1-c}{s})^2 &
(\frac{x_1-c}{s})(\frac{y_1-d}{t})  &\cdots & (\frac{y_1-d}{t})^N \\ 
1 & \frac{x_2-c}{s} & \frac{y_2-d}{t}  &(\frac{x_2-c}{s})^2 &
(\frac{x_2-c}{s})(\frac{y_2-d}{t})  & \cdots & (\frac{y_2-d}{t})^N \\
\vdots & \vdots & \vdots& \vdots &\vdots & \ddots &\vdots \\
1 & \frac{x_{\tilde N}-c}{s} & \frac{y_{\tilde N}-d}{t} &
(\frac{x_{\tilde N}-c}{s})^2 & 
(\frac{x_{\tilde N}-c}{s})(\frac{y_{\tilde N}-d}{t}) & \cdots &
(\frac{y_{\tilde N}-d}{t})^N
\end{pmatrix}\in \R^{\tilde N\times\tilde N}
  \label{for:vand2d}
  \end{align}
is a 2-D Vandermonde matrix and
$f^{(N)}:=\bigl(F(x_1,y_1),F(x_2,y_2),\dots,F(x_{\tilde N},y_{\tilde
N})\bigr)^T\in \R^{\tilde N}$.  A notable feature of polynomial
interpolation in dimensions higher than one is the possibility of
non-uniqueness in the solution to this Vandermonde system (equivalently,
non-uniqueness of $P_N$), even when the collocation points are all distinct.
A nonlinear optimization algorithm for computing well-conditioned
collocation points for polynomial interpolation of order up to $20$ over a
bounded convex domain has been proposed in \cite{vior}. The resulting
points, known as Vioreanu-Rokhlin nodes, are well-conditioned in the sense
that the associated Lebesgue constant is relatively small in magnitude
(which also implies that the corresponding Vandermonde matrix
\eqref{for:vand2d} is invertible).  In Figure \ref{fig:tri}, we plot an
example set of Vioreanu-Rokhlin nodes over a triangle, along with the
corresponding Lebesgue constants for various orders of approximation.

\begin{figure}[h]
    \centering
  \begin{subfigure}{0.49\textwidth}
      \centering
      \includegraphics[width=\textwidth]{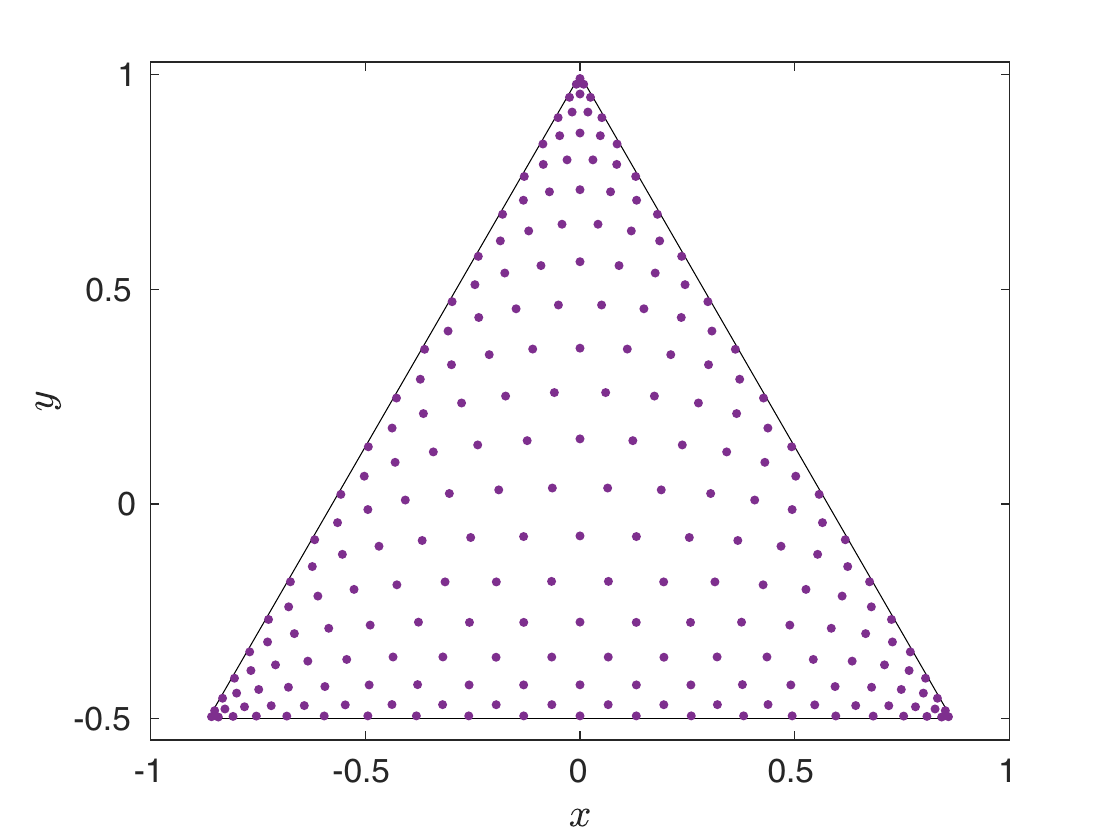}
      \caption{}
      \label{fig:tri:vior}
    \end{subfigure}
    \begin{subfigure}{0.49\textwidth}
      \centering
      \includegraphics[width=\textwidth]{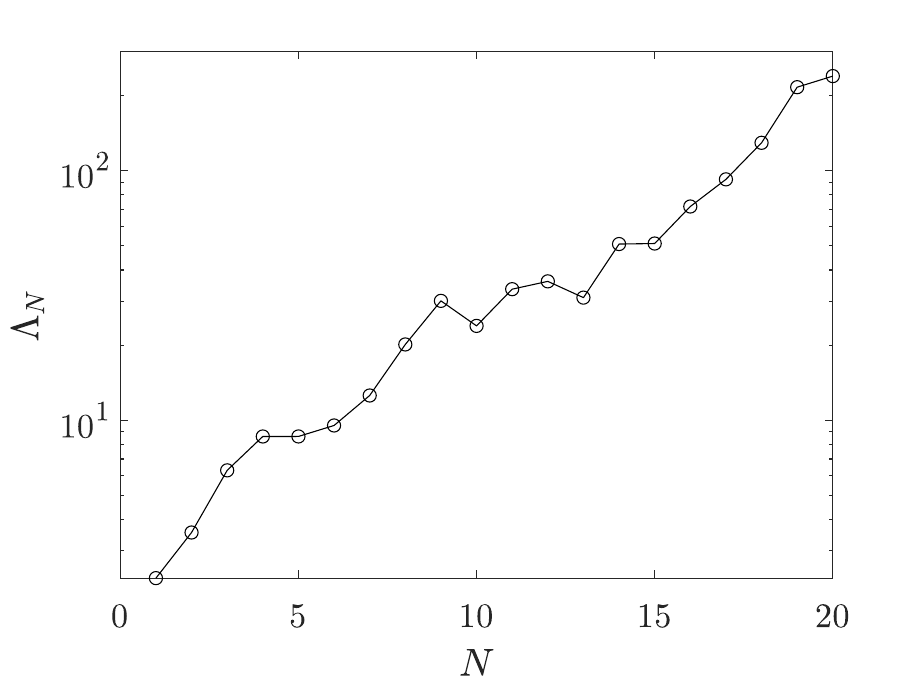}
      \caption{}
      \label{fig:tri:leb}
    \end{subfigure}
  \caption{{\bf The 20th order Vioreanu-Rokhlin nodes over a triangle, and
  the associated Lebesgue constants for various orders of approximation.}.
  The $x$-axis label $N$ denotes the order of approximation. 
  One may observe in Figure \ref{fig:tri:leb} that the Lebesgue constant for
  the Vioreanu-Rokhlin nodes does not exhibit monotonic growth, which is due
  to the heuristic nature of the algorithm used to construct these nodes.
  }
      \label{fig:tri}
\end{figure}

Now let $\tDelta\subset \R^2$ be a star-shaped curved triangle with only one
curved side, $\gamma:[0,L]\to\R^2$ be the parameterization of the
curved side of $\tDelta$, and $\origin\in\R^2$ be the vertex opposite to the
curved side.  The blending function method \cite{blending1} provides a
smooth mapping from the standard simplex $\tri=\{(x,y)\in \R^2 : 0\leq x\leq
1,0\leq y\leq 1-x\}$ to the curved triangle $\tDelta$, defined
by the formula
  \begin{align}
\hspace*{-0em}\rho(\xi,\eta) :=&\ (1-\xi-\eta)\cdot\gamma(L) +
\xi\cdot\gamma(0)+\eta\cdot\origin\notag \\+
&\frac{1-\xi-\eta}{1-\xi}\Bigl(\gamma\bigl(L(1-\xi)\bigr) -
(1-\xi)\cdot\gamma(L)-\xi\cdot\gamma(0)\Bigr).
  \label{for:blend}
  \end{align}
To obtain a set of collocation points over $\tDelta$, we map the
Vioreanu-Rokhlin nodes over $\tri$ to $\tDelta$ via (\ref{for:blend}). 
We observe that the Lebesgue constant of resulting collocation points is
also relatively small in magnitude when $\gamma$ is not too curved.

Similar to the 1-D case, the 2-D Vandermonde matrix \eqref{for:vand2d} is
also exponentially ill-conditioned.  The following theorem provides a priori
bounds for the monomial approximation error, which shows that the accuracy
of approximation is essentially unrelated to the ill-conditioning of
the matrix.  Its proof is almost identical to the proof of Theorem 2.2
in \cite{mono}.
\begin{theorem}
  \label{thm:mono_err}
Let $\Omega\subset \R^2$ be a bounded domain, and let $F:\Omega\to\C$ be an
arbitrary function. Suppose that $P_N$ is the $N$th degree bivariate
interpolating polynomial of $F$ for a given set of $\tilde N$ distinct
collocation points $Z:=\{(x_j,y_j)\}_{j=1,2,\dots,\tilde N}\subset \Omega$. 
Let $V^{(N)}$, $a^{(N)}$ and $f^{(N)}$ be the same as introduced above.  
Suppose that there exists some constant $\gamma_N\geq 0$ such that
the computed monomial coefficient vector $\hat a^{(N)}:= (\hat a_{00},\hat
a_{10},\hat a_{01},\dots,\hat a_{0N})^T\in \R^{\tilde N}$
satisfies 
  \begin{align}
\bigl(V^{(N)}+\delta V^{(N)}\bigr)\hat a^{(N)} = f^{(N)},\label{for:mm0}
  \end{align}
for some $\delta V^{(N)}\in \R^{\tilde N\times \tilde N}$ with
  \begin{align}
\norm{\delta V^{(N)}}_2\leq \mach\cdot
\gamma_N,\label{for:mm1}
  \end{align}
where $u$ denotes machine epsilon.  Let $\hat P_N(x,y):=\sum_{j=0}^N
\sum_{k=0}^{j} \allowbreak\hat a_{j-k,k}
\bigl(\frac{x-c}{s}\bigr)^{j-k}\bigl(\frac{y-d}{t}\bigr)^k$ be the computed
monomial expansion. If the 2-norm of $(V^{(N)})^{-1}$ satisfies
  \begin{align}
\norm{(V^{(N)})^{-1}}_2\leq \frac{1}{2u\cdot \gamma_N},\label{for:mm8}
  \end{align}
then the 2-norm of the numerical solution $\hat a^{(N)}$ is bounded by
  \begin{align}
\frac{2}{3} \norm{a^{(N)}}_2\leq \norm{\hat a^{(N)}}_2\leq
2\norm{a^{(N)}}_2,\label{for:mm999}
  \end{align}
and the monomial approximation error can be quantified a priori by 
  \begin{align}
\hspace*{-0.0em}
\norm{F-\hat P_N}_{L^\infty(\Omega)}
\leq&\,  \norm{F-P_N}_{L^\infty(\Omega)} + 2\mach\cdot \gamma_N\Lambda_N
\norm{a^{(N)}}_2,\label{for:priori1}
  \end{align}
where $\Lambda_N$ denotes the Lebesgue constant for $Z$. 
\end{theorem}

When solving the Vandermonde system using a backward stable linear system
solver, the set of assumptions \eqref{for:mm0} and \eqref{for:mm1} is
satisfied with constant $\gamma_N=\O(\norm{V^{(N)}}_2)$.  Furthermore,
based on the same analysis as in \cite{mono}, one can show that $u\cdot
\norm{a^{(N)}}_2\lesssim \norm{F-P_N}_{L^\infty(\Omega)}$ holds in most
practical situations when $\norm{(V^{(N)})^{-1}}_2$ satisfies the condition
\eqref{for:mm8}, from which it follows that the monomial basis is as good as
an orthogonal polynomial basis for interpolation in such cases.  Therefore,
it is advisable to carefully select the monomial expansion center $(c,d)$
and the scaling factors $s$, $t$ to minimize the growth of both
$\norm{V^{(N)}}_2$ and $\norm{(V^{(N)})^{-1}}_2$.  Below, we provide an
algorithm for choosing these constants.

Given an arbitrary bounded domain $\Omega$ in $\mathbb{R}^2$, we define
$B$ to be the minimum bounding box of $\Omega$ (see Figure
\ref{fig:cond:1}). Then, we establish a local coordinate system centered at
the midpoint of $B$, with the $x$- and $y$-axes aligned parallel to the
sides of $B$. In this coordinate system, we set $c$ and $d$ to be zero, $s$
to be half the length of the longer side of $B$, and $t$ to be half the
length of the shorter side of $B_1$.  One can
show that the entries of the resulting Vandermonde matrix $V^{(N)}$ are no
larger than one in magnitude, which implies that the constant $\gamma_N$ is
small. In addition, one can observe from Figure \ref{fig:cond:2}
that the condition \eqref{for:mm8} is satisfied for $N\lesssim 20$,
regardless of the triangle's aspect ratio. 

In Section \ref{sec:bimo}, we provide numerical experiments to demonstrate
the feasibility of bivariate polynomial interpolation in the monomial basis.

\begin{figure}
    \centering
    \begin{subfigure}{0.49\textwidth}
      \centering
      \includegraphics[width=\textwidth]{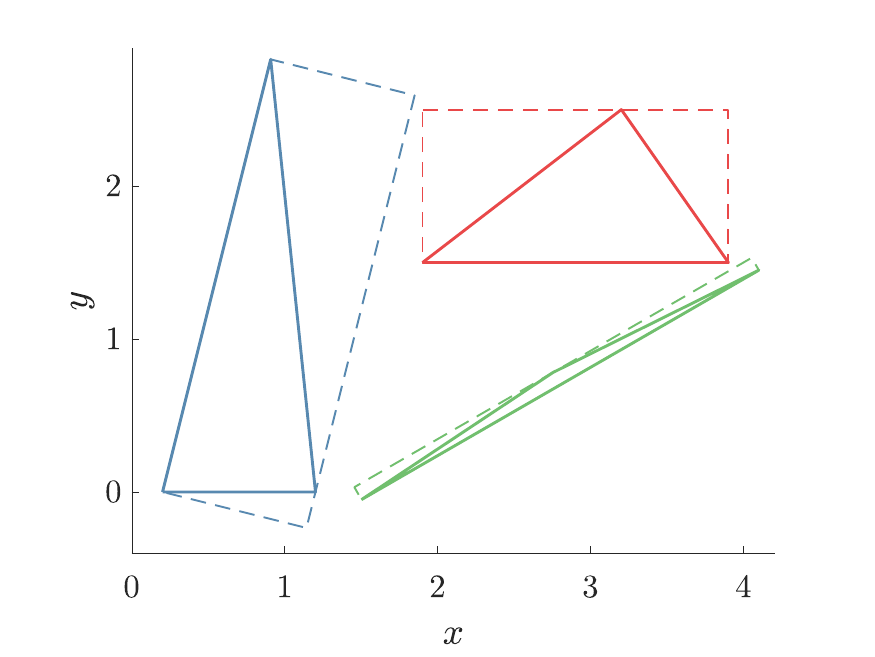}
      \caption{\label{fig:cond:1}}
    \end{subfigure}
    \begin{subfigure}{0.49\textwidth}
      \centering
      \includegraphics[width=\textwidth]{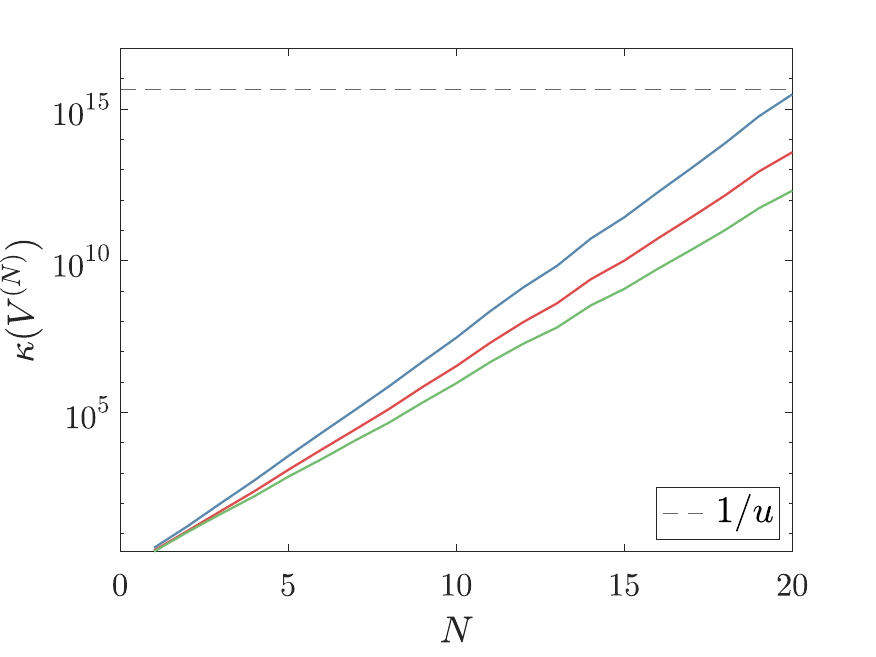}
      \caption{\label{fig:cond:2}}
    \end{subfigure}
  \caption{{\bf The growth of $\kappa(V^{(N)})$ for triangles
  with different aspect ratios}. The colors of the triangles in Figure
  \ref{fig:cond:1} correspond to the line colors depicted in
  Figure~\ref{fig:cond:2}. The boxes in Figure \ref{fig:cond:1} define
  the local coordinate for each triangle.}
  \label{fig:cond}
\end{figure}

\section{Numerical algorithm}
In this section, we first present an algorithm for computing the Newtonian
potential when the domain $\Omega$ is a (possibly curved)
triangle.  Then, we describe how to apply this algorithm to compute the
Newtonian potential over a general domain. In the end, we show
that our algorithm has linear time complexity.

\subsection{Construction of the anti-Laplacian mapping}
\label{sec:antilap}
In this section, we present an algorithm for computing the anti-Laplacian of
a bivariate monomial.  With a slight abuse of notation, we denote the
anti-Laplacian operator by $\nabla^{-2}$, which we define by the recurrence
relations
  \begin{subequations}
  \begin{align}
\nabla^{-2} \Bigl[\Bigl(\frac{x-c}{s}\Bigr)^m\Bigr] =&\,
\frac{s^2}{(m+1)(m+2)}\Bigl(\frac{x-c}{s}\Bigr)^{m+2},\\
\nabla^{-2} \Bigl[\Bigl(\frac{x-c}{s}\Bigr)^m
\Bigl(\frac{y-d}{t}\Bigr)\Bigr]
=&\,\frac{s^2}{(m+1)(m+2)}\Bigl(\frac{x-c}{s}\Bigr)^{m+2}\Bigl(\frac{y-d}{t}\Bigr),
  \end{align}
  \end{subequations}
for all $m\geq 0$, and
  \begin{subequations}
  \begin{align}
\nabla^{-2} \Bigl[\Bigl(\frac{y-d}{t}\Bigr)^n\Bigr] =&\,
\frac{t^2}{(n+1)(n+2)}\Bigl(\frac{y-d}{t}\Bigr)^{n+2},\\
\nabla^{-2}
\Bigl[\Bigl(\frac{x-c}{s}\Bigr)\Bigl(\frac{y-d}{t}\Bigr)^n\Bigr]
=&\,\frac{t^2}{(n+1)(n+2)}\Bigl(\frac{x-c}{s}\Bigr)
\Bigl(\frac{y-d}{t}\Bigr)^{n+2},
  \end{align}
  \end{subequations}
for all $n\geq 0$, and
  \begin{subequations}
  \begin{align}
  \hspace*{-6.5em}
\nabla^{-2}
\Bigl[\Bigl(\frac{x-c}{s}\Bigr)^m\Bigl(\frac{y-d}{t}\Bigr)^n\Bigr]= &\,
\frac{s^2}{(m+2)(m+1)}\Bigl(\frac{x-c}{s}\Bigr)^{m+2}\Bigl(\frac{y-d}{t}\Bigr)^n\notag\\
&\,-\frac{s^2n(n-1)}{t^2(m+2)(m+1)}\nabla^{-2}
\Bigl[\Bigl(\frac{x-c}{s}\Bigr)^{m+2}\Bigl(\frac{y-d}{t}\Bigr)^{n-2}\Bigr]\label{for:antix}\\ 
=&\, 
\frac{t^2}{(n+2)(n+1)}\Bigl(\frac{x-c}{s}\Bigr)^{m}\Bigl(\frac{y-d}{t}\Bigr)^{n+2}\notag\\
&\,-\frac{t^2m(m-1)}{s^2(n+2)(n+1)}\nabla^{-2}
\Bigl[\Bigl(\frac{x-c}{s}\Bigr)^{m-2}\Bigl(\frac{y-d}{t}\Bigr)^{n+2}\Bigr],\label{for:antiy} 
  \end{align}
  \end{subequations}
for all $m\geq 2,n\geq2$. It is easy to verify that
$\nabla^2\circ\nabla^{-2}=I$.   Based on these recurrence relations, one can
construct a mapping from the monomial coefficients of a bivariate polynomial
of degree $N$ to the monomial coefficients of the anti-Laplacian of this
polynomial. This mapping only has $\O(N^2)$ non-zero entries, and should be
stored in a sparse format. 

\begin{remark}
The identity (\ref{for:antix}) produces a shorter sequence of recurrence
relations when $m\geq n$, and vice versa for (\ref{for:antiy}).
\end{remark}

\subsection{Close and self-evaluation of Newtonian potential over a
mesh element}
\label{sec:neareval}

Given a possibly curved triangle $\bDelta$ and a function $f:\bDelta \to\R$,
we first compute its bivariate interpolating polynomial in the monomial
basis, as described in Section \ref{sec:mono}.  By Green's third identity
(see Theorem \ref{thm:green3}), the Newtonian potential with density
function $f$ over $\bDelta$ at a given target $x\in\R^2$ can be expressed as
  \begin{align}
\hspace*{-6.0em}\iint_{\bDelta} G(x,y)f(y)\d A_y \approx &\,\iint_{\bDelta}
G(x,y)P_N(y)\d A_y\notag\\
=&\,\varphi(x)\mathbbm{1}_{\bDelta} (x)+\oint_{\partial \bDelta} \Bigl(
G(x,y)\pp{\varphi}{n_y}(y) - \pp{G(x,y)}{n_y} \varphi(y) \Bigr)\d \ell_y,
\label{for:g1d}
  \end{align}
where $P_N$ is the $N$th degree bivariate interpolating polynomial
of $f$, and $\varphi:=\nabla^{-2}[P_N]$ is a bivariate polynomial of
degree $(N+2)$ computed using the algorithm
outlined in the previous section. Thus, it remains to compute the layer
potentials $\int_{L_i} G(x,y)\pp{\varphi}{n_y}(y) \d \ell_y$ and $\int_{L_i}
\pp{G(x,y)}{n_y} \varphi(y) \d \ell_y$ for $i=1$, $2$, $3$, where $L_i$
denotes the $i$th edge of $\bDelta$. When $x$ is well-separated from $L_i$,
these integrands are smooth, and \eqref{for:g1d} is computed using
a standard quadrature rule. When $x$ is close to $L_i$, these integrands
become nearly-singular, and the Helsing-Ojala method is used
for the calculation (see Section \ref{sec:helsing}). We note that the
restriction of the anti-Laplacian $\varphi$ to a line segment is a
univariate polynomial of degree $N+2$.

\begin{remark}
It is well-known that Horner's method evaluates a polynomial in the monomial
basis with the fewest number of multiplications.  However, Estrin's scheme
outperforms Horner's method in terms of speed on a modern computer, as it
effectively utilizes CPU pipelines.
\end{remark}

\subsection{Generalization to an arbitrary domain}
Given a general planar region $\Omega$, we first discretize $\Omega$ into
triangles and curved triangles using a standard off-the-shelf meshing
algorithm.  Then, we construct the anti-Laplacian of the density function in
the form of a 2-D monomial expansion for each mesh element.  We also compute
the 1-D monomial expansion coefficients for the restriction of the
anti-Laplacian and its normal derivatives to the edges of each element. At
this stage, all of the required precomputations are completed.

Then, we use the point-based fast multipole method (FMM) \cite{fmm} to
compute the far field interactions. Typically, one uses 2-D quadrature rules
over triangles to compute the far field interactions generated over mesh
elements (see, for example, \cite{volint,anderson}) and, thus, the number of
quadrature nodes over each element for computing far field interactions is
of order $\O(N^2)$, where $N$ is the degree of the bivariate interpolating
polynomial for the density function. We note, however, that in the far
field, the layer potentials in Green's third identity (\ref{for:green3}) can
be computed efficiently and accurately by Gauss-Legendre rules. It follows
that, in our algorithm, the number of quadrature nodes over each element is
of order $\O(N)$. Finally, we compute the near and
self-interactions using the algorithm presented in the previous section, and
use the ``subtract-and-add'' method to remove the spurious contribution from
the FMM (see \cite{leslie} for details).

\begin{remark}
\label{rem:edge}
Given two adjacent mesh elements, their far field quadrature nodes over their
common edge coincide, and thus, one could merge the nodes in the FMM
computation to reduce the number of sources by a factor of two. Similarly,
one could merge the expansion coefficients of the two 1-D monomial
expansions over the common edge of two elements to reduce the near
interaction computational cost by a factor of two.
\end{remark}

\subsection{Time complexity analysis}
In this section, we present the time complexity of our algorithm. 
Suppose that we construct a $p$th degree bivariate interpolating polynomial
of the density function over each element.  Consequently, each polynomial is
represented by $(p+1)(p+2)/2=\O(p^2)$ terms in the monomial basis.  We
estimate the various costs as follows.

\paragraph{Precomputation for each mesh element:} 
\begin{enumerate}
\item The computation of the 2-D monomial expansion coefficients of the
density function takes $\O(p^6)$ operations, since the cost is dominated by
the factorization of a 2-D Vandermonde matrix of size $\O(p^2)\times
\O(p^2)$.
\item The computation of the anti-Laplacian takes $\O(p^3)$ operations.
\item The evaluation of the anti-Laplacian and its normal derivative at the
$\O(p)$ collocation points on the edges of the mesh elements takes $\O(p^3)$
operations.
\item 
The computation of the 1-D monomial expansions of the restriction of the
anti-Laplacian and its normal derivatives takes $\O(p^2)$ operations on a
straight edge (as one can store and reuse the pivoted LU factorization of
the 1-D Vandermonde matrix with Gauss-Legendre collocation nodes over
$[-1,1]$), and takes $\O(p^3)$ on a curved edge.  We note that the total
number of curved edges is generally far fewer than the total number of
straight edges.
\end{enumerate}

\paragraph{Close and self-evaluation of the Newtonian potential over a mesh
element at a single target location:}
\begin{enumerate}
\item It takes $\O(p)$ operations to evaluate the layer potentials on
the right hand side of Green's third identity (\ref{for:green3}), either by
the Gauss-Legendre rule or the Helsing-Ojala method.
\item It takes $\O(p^2)$ operations to evaluate the anti-Laplacian on the
right hand side of Green's third identity (\ref{for:green3}). This is only
required by the self-evaluation.
\end{enumerate}

Based on these estimates, we present the total number of
operations required to evaluate the Newtonian potential over all of the
discretization nodes over the domain $\Omega$. Suppose that $\Omega$ is
discretized into $m$ mesh elements.  Then, the number of discretization
nodes $N_{\text{tot}}$ is of order $\O(p^2m)$.  First, the precomputation
takes $\O(p^6m)=\O(p^4N_{\text{tot}})$
operations.  Second, the far field interaction costs (i.e., the FMM cost) are
of order $\O(pm+N_{\text{tot}})$. Third, the near interaction computation
takes $\O(pN_{\text{tot}})$ operations, as each discretization node is
inside the near fields of a constant number of mesh elements. Finally, the
self-interaction computation takes $\O(p^2N_{\text{tot}})$ operations. Since
$p$ is generally a small constant, our algorithm has linear
time complexity.  Furthermore, we note that the constant associated with the
precomputation is small (see Table \ref{tab:poi1}), and the near and
self-interaction computations are nearly instantaneous after the
precomputations have been performed.

\section{Numerical experiments}
In this section, we illustrate the performance of the algorithm with several
numerical examples.  We implemented our algorithm in Fortran 77 and Fortran 90, and
compiled it using the Intel Fortran Compiler, version 2021.6.0, with the
\texttt{-Ofast} flag. We conducted all experiments on a ThinkPad laptop,
with 16 GB of RAM and an Intel Core i7-10510U CPU. 

We use the Vioreanu-Rokhlin rules \cite{vior}, which are publicly
available in \cite{triasymq}. We use the FMM library published in
\cite{fmm2d} in our implementation. We use the subroutines \texttt{dgetrf} and
\texttt{dgetrs} (i.e., LU factorization with partial pivoting) from LAPACK
as our linear system solver for the Vandermonde system.  We make no use of
parallelization. While we comment in Remark \ref{rem:edge} that it is more
efficient to loop through edges instead of triangles, these features are not
implemented in our code, for the sake of simplicity.

We list the notation that appears in this section below.
\begin{itemize}
\item $S_{\text{exps}}$: The number of targets at which the Newtonian
potential generated over a mesh element can be evaluated per second using
our algorithm, after the precomputation.
\item $S_{\text{adap}}$: The number of targets at which the Newtonian
potential generated over a mesh element can be evaluated, per second, using
adaptive integration.
\item $E_{\text{exps}}$: The absolute error of the potential evaluation computed
using our algorithm.
\item $E_{\text{adap}}$: The absolute error of the potential evaluation computed
using adaptive integration.
\item $N_{\text{elem}}$: The number of mesh elements.
\item $h_0^{\max}$:  The maximum diameter of all mesh elements. 
\item $h_0^{\min}$:  The minimum diameter of all mesh elements. 
\item $A^{\max}$:  The maximum aspect ratio of all mesh
elements.  The aspect ratio of a triangle equals the ratio of its
circumradius to twice its inradius.
\item $N_{\text{ord}}$: The order of the bivariate polynomial approximation
to the density function over each mesh element.
\item $N_{\text{tot}}^{\text{tgt}}$: The total number of targets at which
the Newtonian potential is evaluated.
\item $N_{\text{tot}}^{\text{src}}$: The total number of sources (i.e.,
far field quadrature nodes).
\item $T_{\text{geom}}$: The time spent on the geometric algorithms: 
quadtree constructions, nearby elements queries, etc. The mesh creation time
is not counted.
\item $T_{\text{init}}$: The time spent on the precomputations required by our
algorithm.
\item $T_{\ft}$: The time spent on the far field interaction computation (i.e.,
the FMM computation).
\item $T_{\nt}$: The time spent on the near field interaction computation,
including the subtraction of spurious contributions from the far field
interaction computation (see \cite{leslie,volint}).
\item $T_{\st}$: The time spent on the self-interaction computation, including
the subtraction of spurious contributions from the far field
interaction computation (see \cite{leslie,volint}).
\item $T_{\text{tot}}$: The total time for the evaluation of the volume potential at all
of the discretization nodes.
\item $\frac{\#\text{tgt}}{\text{sec}}$: 
The number of targets at which the Newtonian potential can be evaluated per
second using our algorithm, including the precomputation cost.
\item $E_{\text{poi}}$: The largest absolute error of the solution to
Poisson's equation at all of the target nodes.
\item $E_{\text{den}}$: The $L^\infty$ monomial approximation
error of the density function over the whole domain, estimated by comparing
the approximated function values at 40,000 uniformly sampled points over
the domain with the true function values.
\item $E_{\text{tol}}$: The error tolerance of the adaptive version of our 
algorithm. Used in the experiment with an inhomogeneous density function.
\end{itemize}

\subsection{Bivariate polynomial interpolation in the monomial basis}
\label{sec:bimo}

In this section, we present numerical experiments to demonstrate the
feasibility of bivariate polynomial interpolation in the monomial basis.

Let $\Delta_1$ be an equilateral triangle with vertices
$(-1,0),(1,0)$ and $(0, \sqrt{3})$, and let $\Delta_2$ be a flattened triangle
with vertices $(-1,0),(1,0)$ and $(0, \frac{1}{16})$.  In Figures
\ref{fig:taylor1} and \ref{fig:taylor2}, we estimate $L^\infty$ errors of
bivariate polynomial interpolation in the 2-D monomial basis and in the
Koornwinder polynomial basis \cite{bremer1} (i.e., the orthogonal polynomial
basis over a triangle) over the domains $\Delta_1$ and $\Delta_2$ with the
Vioreanu-Rokhlin collocation points, for different orders of approximation
$N$, by comparing the values of the two approximations with the true values
of the function at 20,000 uniformly sampled points over the domain.  We also
report the extra numerical error that arises from the use of monomial basis
(i.e., $u\cdot \norm{a^{(N)}}_2$; see Theorem \ref{thm:mono_err}).  One can
observe that the use of the monomial basis induces essentially no extra loss
of accuracy in both cases.

\begin{figure}[h]
    \centering
    \begin{subfigure}{0.49\textwidth}
      \centering
      \includegraphics[width=\textwidth]{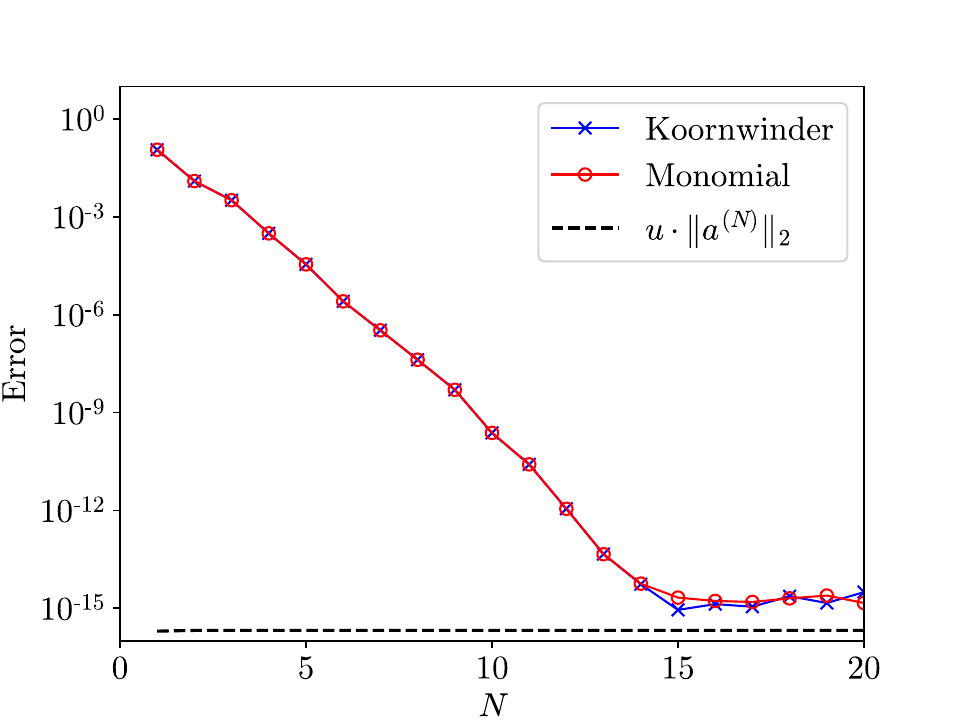}
      \caption{$f(x,y)=e^{-(x^2+y^2)/8}$\label{fig:taylor1_a}}
    \end{subfigure}
    \begin{subfigure}{0.49\textwidth}
      \centering
      \includegraphics[width=\textwidth]{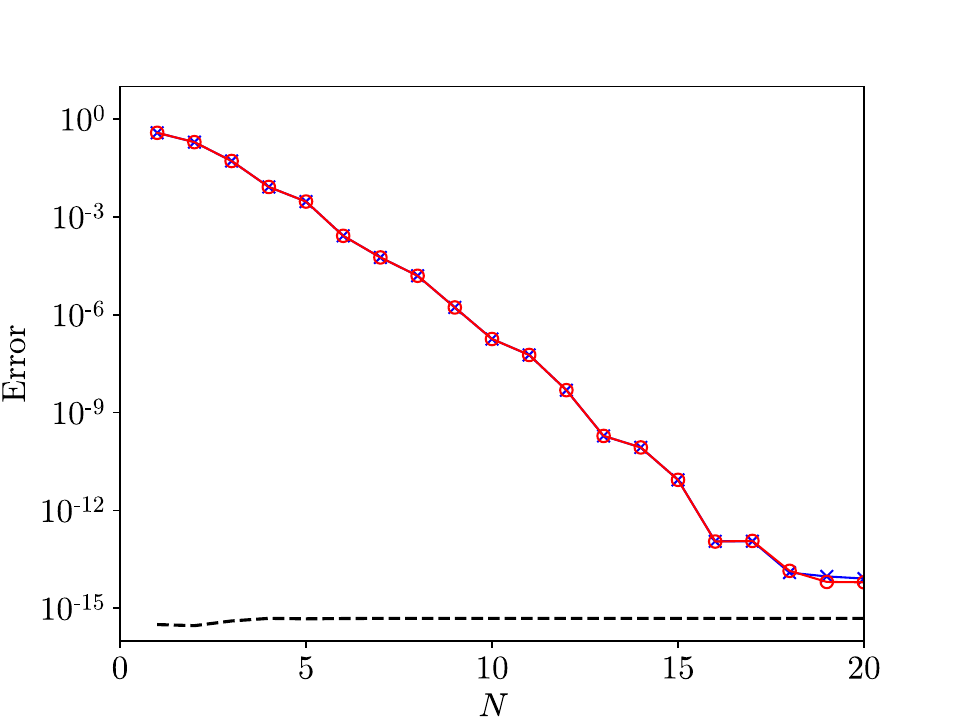}
      \caption{$f(x,y)=\sin(\frac{1}{2}xy+x+y)$\label{fig:taylor1_b}}
    \end{subfigure}
  \begin{subfigure}{0.49\textwidth}
    \centering
    \includegraphics[width=\textwidth]{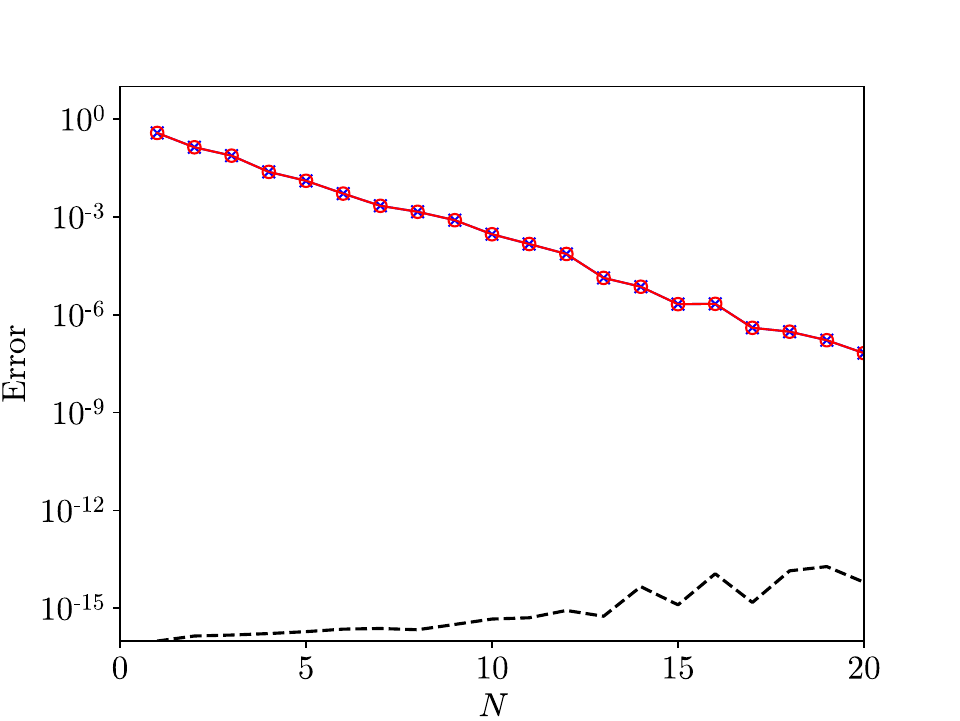}
    \caption{$f(x,y)=\frac{1}{x^2+(y+1)^2}$\label{fig:taylor1_c}}
  \end{subfigure}
  \begin{subfigure}{0.49\textwidth}
    \centering
    \includegraphics[width=\textwidth]{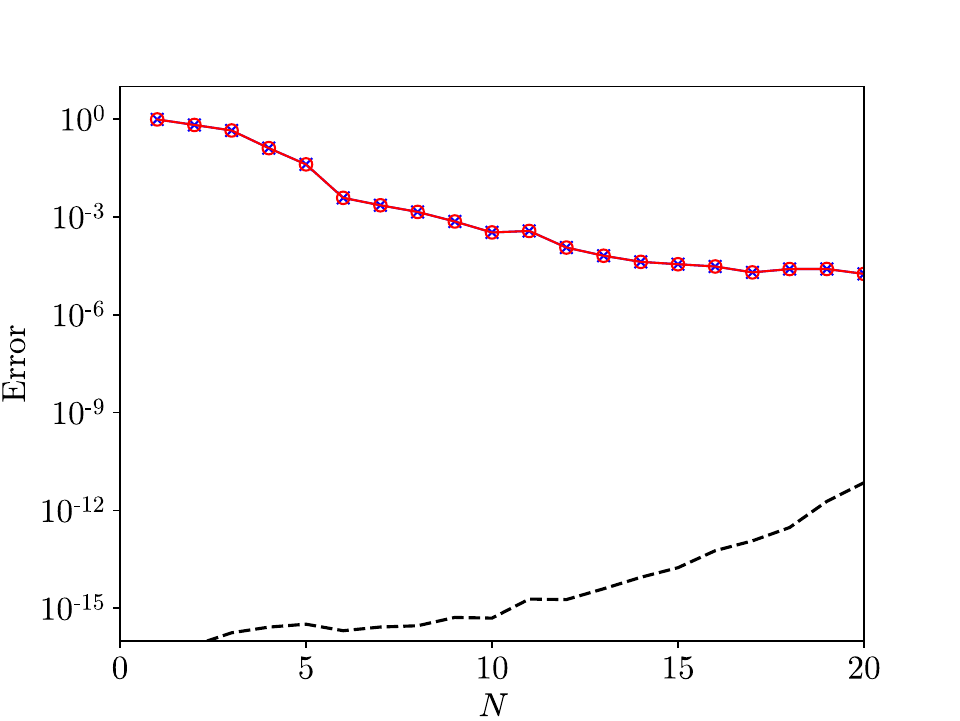}
    \caption{$f(x,y)=\abs{x}^{5.5}$}
  \end{subfigure}

  \caption{{\bf Bivariate polynomial interpolation over the equilateral
  triangle $\Delta_1$ by monomials and Koornwinder polynomials}. }
  \label{fig:taylor1}
\end{figure}

\begin{figure}[h]
    \centering
    \begin{subfigure}{0.49\textwidth}
      \centering
      \includegraphics[width=\textwidth]{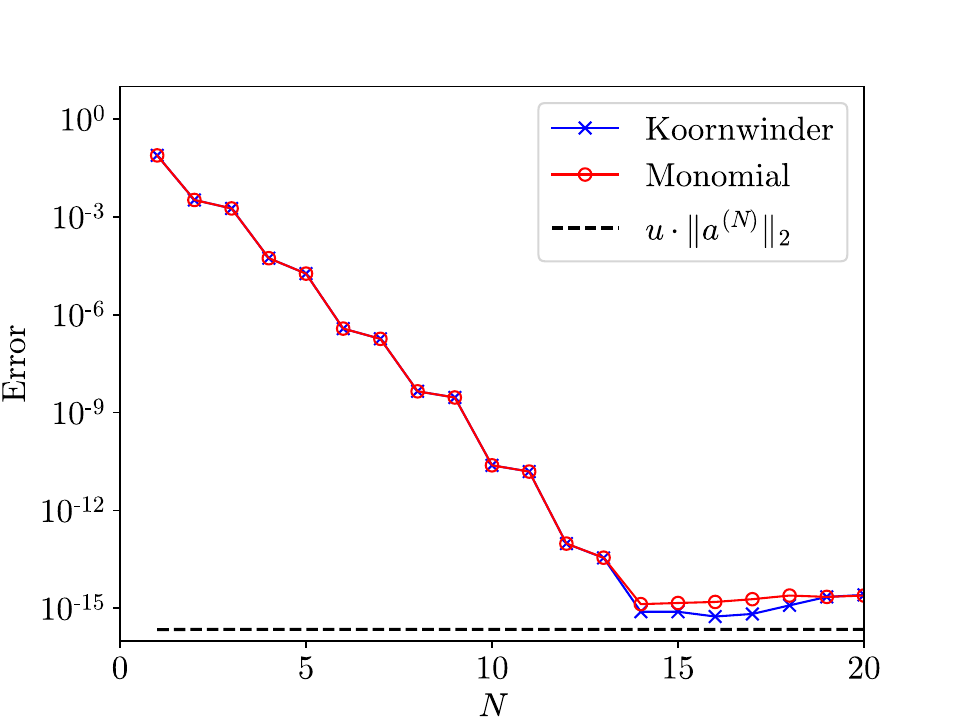}
      \caption{$f(x,y)=e^{-(x^2+y^2)/8}$}
    \end{subfigure}
    \begin{subfigure}{0.49\textwidth}
      \centering
      \includegraphics[width=\textwidth]{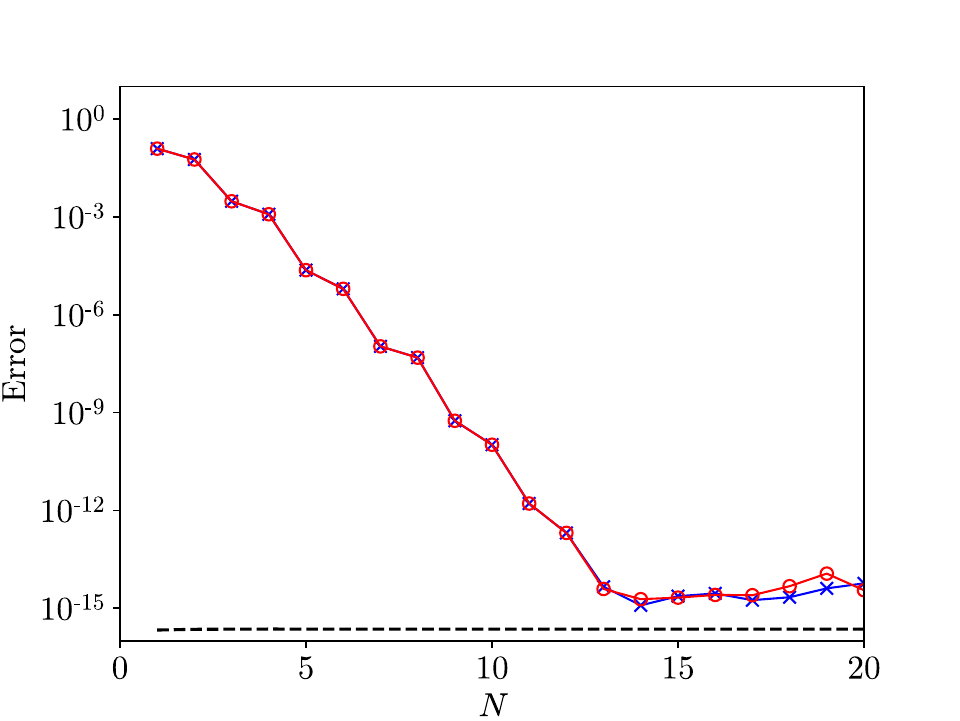}
      \caption{$f(x,y)=\sin(\frac{1}{2}xy+x+y)$}
    \end{subfigure}
  \begin{subfigure}{0.49\textwidth}
    \centering
    \includegraphics[width=\textwidth]{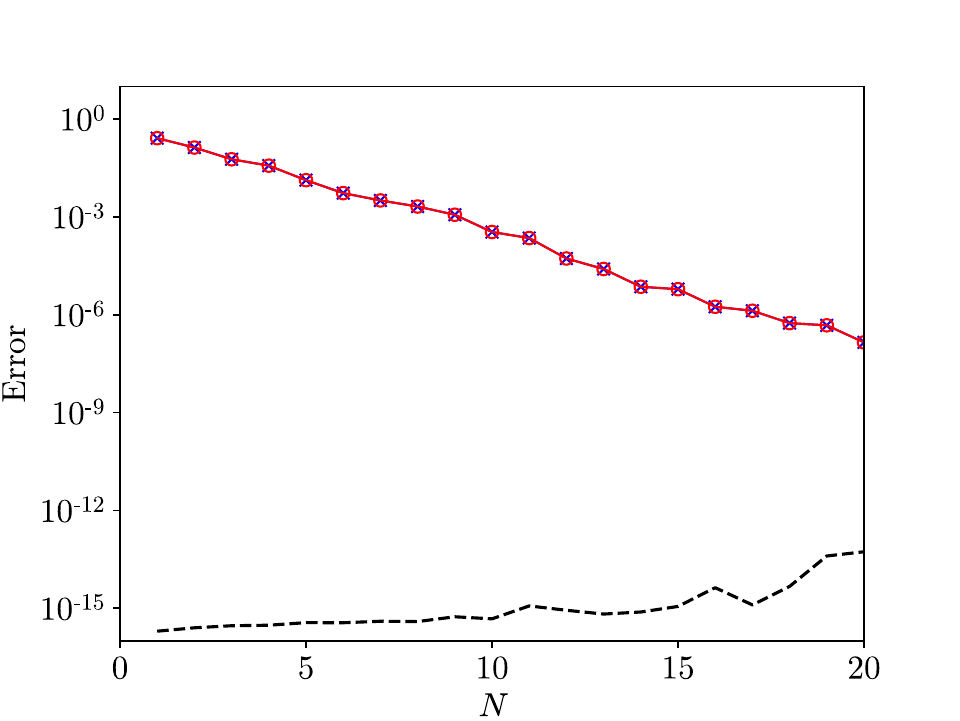}
    \caption{$f(x,y)=\frac{1}{x^2+(y+1)^2}$}
  \end{subfigure}
  \begin{subfigure}{0.49\textwidth}
    \centering
    \includegraphics[width=\textwidth]{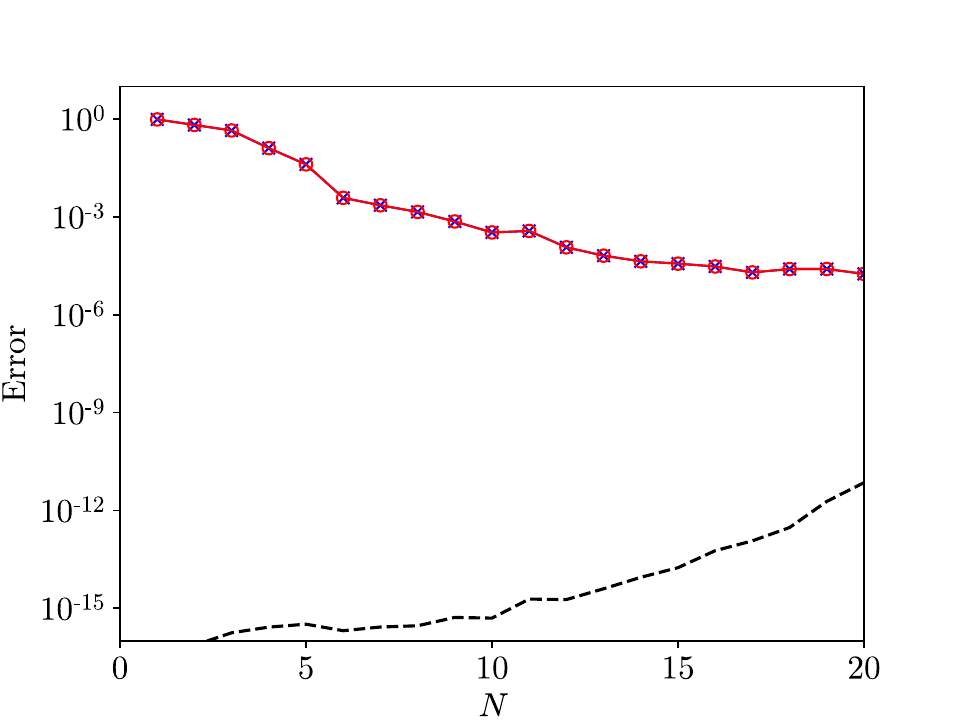}
    \caption{$f(x,y)=\abs{x}^{5.5}$}
  \end{subfigure}

  \caption{{\bf Bivariate polynomial interpolation over the flattened 
  triangle $\Delta_2$ by monomials and Koornwinder polynomials}. }
  \label{fig:taylor2}
\end{figure}

Now let $\tDelta$ be a curved triangle, given by the formula
  \begin{align}
\tDelta:=\{(r\cos\theta-1,r\sin\theta)\in \R^2 : 0\leq r\leq
2,0\leq \theta\leq \pi/3\}.
  \label{for:experiment_curv}
  \end{align}
We repeat the previous experiment on $\tDelta$, and show the results   
in Figure \ref{fig:ctaylor1}. One can observe that the performance of the
monomial approximation over a curved triangle is almost identical to the
triangle domain case.

\begin{figure}[h]
    \centering
    \begin{subfigure}{0.49\textwidth}
      \centering
      \includegraphics[width=\textwidth]{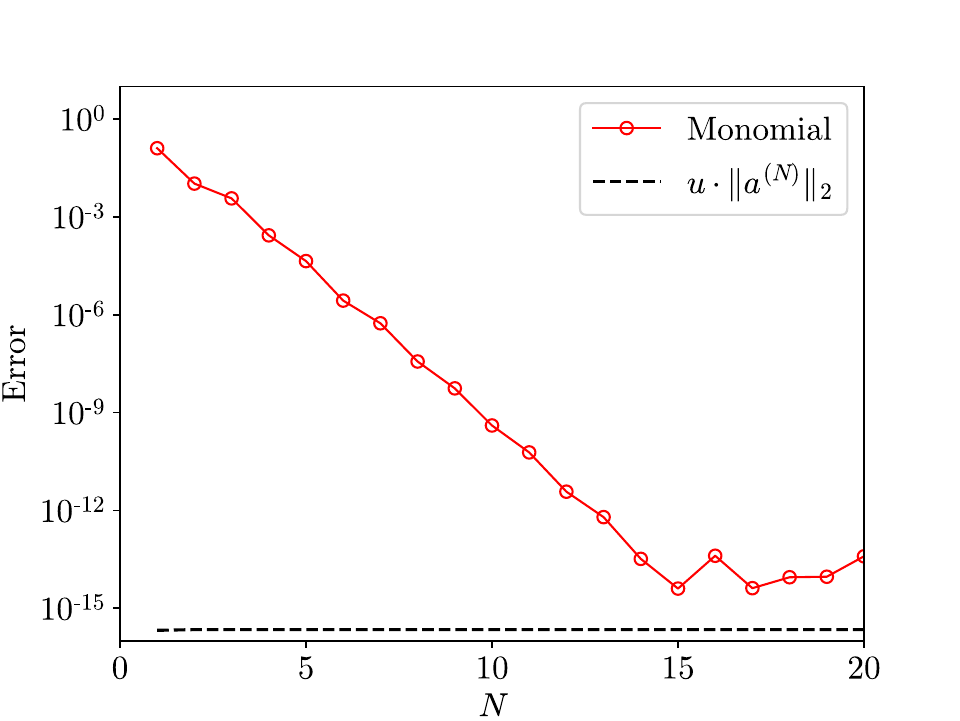}
      \caption{$f(x,y)=e^{-(x^2+y^2)/8}$\label{fig:ctaylor1_a}}
    \end{subfigure}
    \begin{subfigure}{0.49\textwidth}
      \centering
      \includegraphics[width=\textwidth]{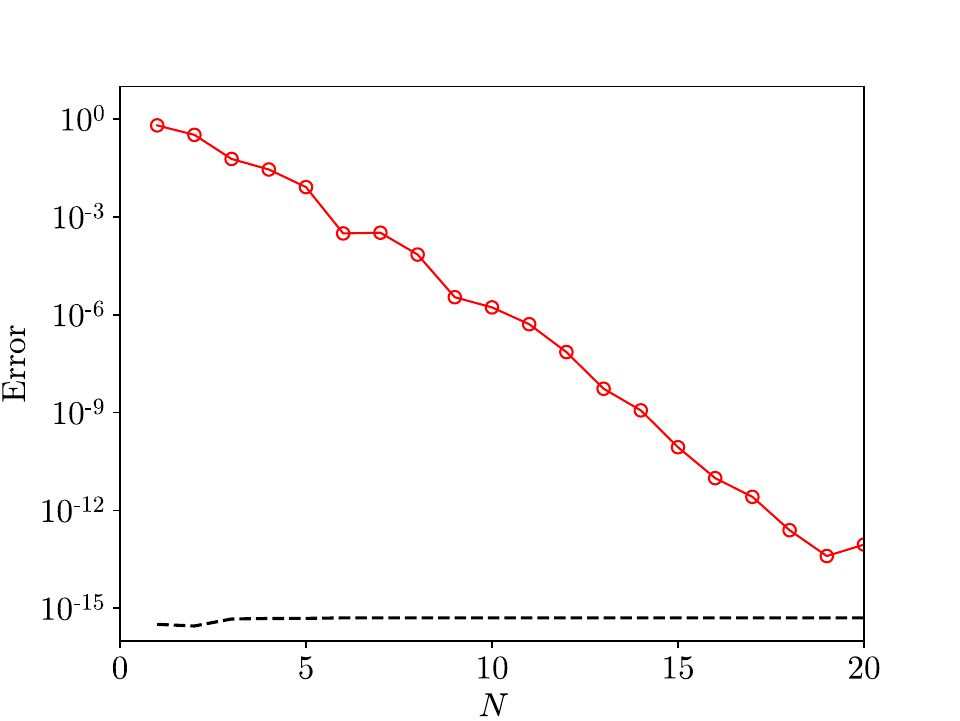}
      \caption{$f(x,y)=\sin(\frac{1}{2}xy+x+y)$\label{fig:ctaylor1_b}}
    \end{subfigure}
  \begin{subfigure}{0.49\textwidth}
    \centering
    \includegraphics[width=\textwidth]{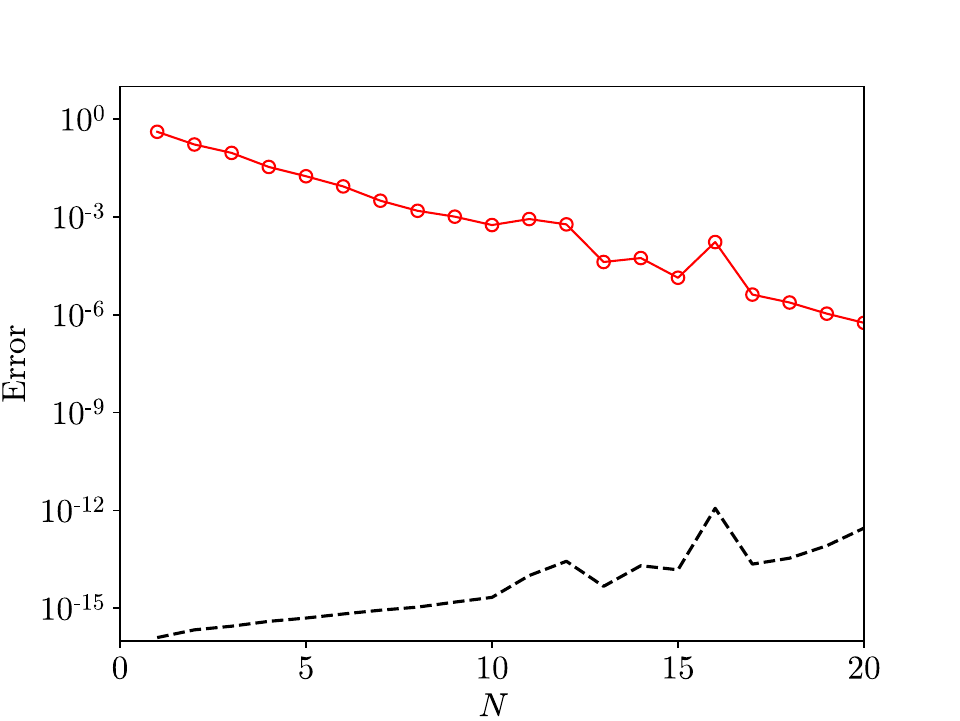}
    \caption{$f(x,y)=\frac{1}{x^2+(y+1)^2}$}
  \end{subfigure}
  \begin{subfigure}{0.49\textwidth}
    \centering
    \includegraphics[width=\textwidth]{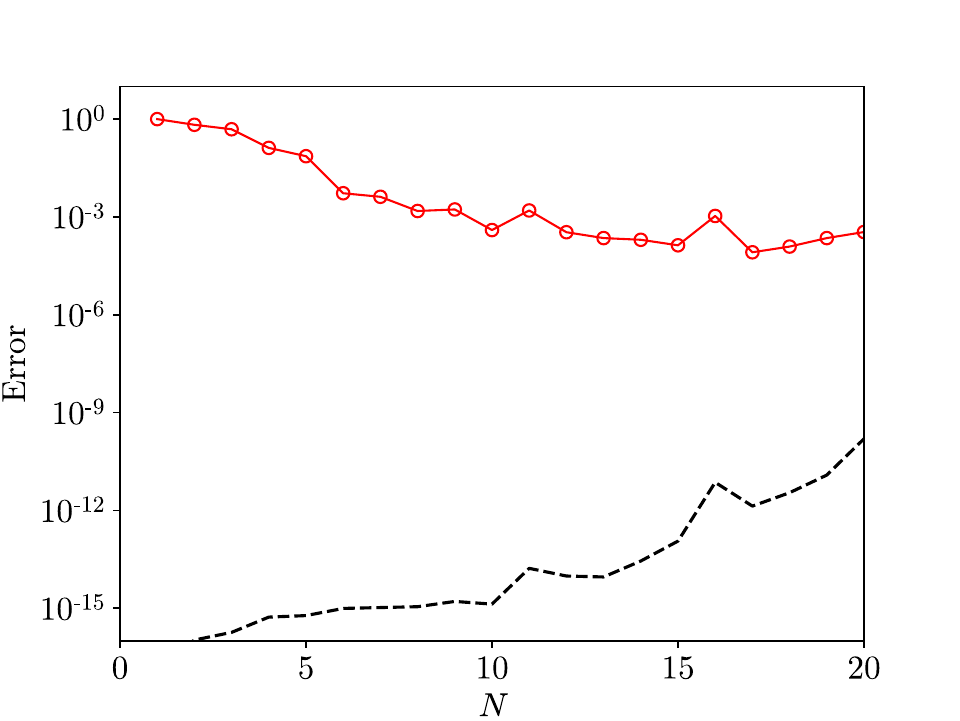}
    \caption{$f(x,y)=\abs{x}^{5.5}$}
  \end{subfigure}

  \caption{{\bf Bivariate polynomial interpolation over the curved triangle
  element~$\tDelta$ by monomials}.}
  \label{fig:ctaylor1}
\end{figure}

\subsection{Newtonian potential generated over a mesh element}

In this section, we compare our algorithm with adaptive integration for 
Newtonian potential evaluation.
We set the domain to be $\Delta:=\{(x,y)\in \R^2 : 0\leq
x\leq 1,0\leq y\leq 1-x\}$, the density function to be
$f(x,y)=\cos(5xy)+\sin(2x+1)+\cos(3y-1)$, and the target location to be
$(0.5,-h)$, for various $h$.  In the adaptive integration computation, we
used the Vioreanu-Rokhlin rule of order $N_{\text{ord}}$, equipped with the
error control technique introduced in \cite{volint} to align the error with
the error of our algorithm. To make the adaptive integration speed
independent of the complexity of the density function, we excluded the time
spent on the density function evaluations on the first level, but included
the time spent on interpolating the density function values to the next
level (see \cite{bremer1} for a fast interpolation technique).  Furthermore,
we accelerated the application of the interpolation matrix by LAPACK, and
fine-tuned the baseline to make the comparisons fair.  The reference
solutions were computed using extremely high-order adaptive integration.  In
Table \ref{tab:near_speedup}, we report the speed-up of the close evaluation
of the Newtonian potential obtained by our algorithm after the
precomputation, compared to the conventional adaptive integration-based
approach, for different orders of approximation. We do not include a similar
experiment that demonstrates the speed-up of the self-evaluation, since a
fair experiment requires the speed of the adaptive integration-based
approach to be independent of the cost of evaluating the density function.
We note that, even when the density function is a constant (so that its
evaluation is free), the speed of the self-evaluation by our algorithm is
significantly faster than the adaptive integration-based approach when the
target point is close to the boundary of the domain.  In Figure
\ref{fig:speed1}, we report the speeds for the close and self-evaluations,
and of the precomputations, for various orders of approximations.  In Figure
\ref{fig:pot_heat}, we show the plots of the computational errors of the
Newtonian potential.

\begin{table}[h!!]
    \begin{center}
    \begin{tabular}{cc ccc cc}
     $N_{\text{ord}}$ & $h$ & $S_{\text{exps}}$ & $S_{\text{adap}}$ &
     $\frac{S_{\text{exps}}}{S_{\text{adap}}}$& $E_{\text{exps}}$
     &$E_{\text{adap}}$\\
    \midrule
    \addlinespace[.5em]
    8& $2\e{-1}$   & 9.29\e{5} & 3.56\e{5} & 2.61 & 4.07\e{-8}  & 1.37\e{-7}\\
    \addlinespace[.25em]
     &  $2\e{-2}$   & 1.19\e{6} & 2.02\e{5} & 5.88 &  3.06\e{-8} &9.73\e{-8}\\
    \addlinespace[.25em]
    &  $2\e{-3}$   & 1.19\e{6} & 6.13\e{4} & 19.4 & 4.89\e{-8} &2.78\e{-7}\\
    \addlinespace[.25em]
    &  $2\e{-4}$   & 1.19\e{6} & 5.39\e{4}& 22.1 &  5.10\e{-8} & 5.35\e{-8}\\
    \addlinespace[.25em]
     &  $2\e{-5}$   & 1.19\e{6} & 3.95\e{4} & 30.1 &  5.12\e{-8} & 7.03\e{-8}\\
    \addlinespace[.5em]
    14& $2\e{-1}$   & 1.04\e{6} & 5.07\e{4} & 20.6 & 9.42\e{-13}  & 9.41\e{-13}\\
    \addlinespace[.25em]
     &  $2\e{-2}$   & 1.04\e{6} & 1.51\e{4} & 69.0 &  1.69\e{-11} &1.20\e{-11}\\
    \addlinespace[.25em]
    &  $2\e{-3}$   & 1.04\e{6} & 8.21\e{3} & 127 & 2.27\e{-11} &5.83\e{-11}\\
    \addlinespace[.25em]
    &  $2\e{-4}$   & 1.04\e{6} & 5.64\e{3}& 185 &  2.34\e{-11} & 5.30\e{-11}\\
    \addlinespace[.25em]
     &  $2\e{-5}$   & 1.04\e{6} & 4.09\e{3} & 255 &  2.35\e{-11} & 5.66\e{-11}\\
    \addlinespace[.5em]
    20& $2\e{-1}$   & 7.40\e{5} & 1.19\e{4} & 62.2 & 7.77\e{-16}  & 1.00\e{-16}\\
    \addlinespace[.25em]
     &  $2\e{-2}$   & 7.41\e{5} & 3.11\e{3} & 238 &  4.16\e{-16} &8.33\e{-17}\\
    \addlinespace[.25em]
    &  $2\e{-3}$   & 7.41\e{5} & 1.56\e{3} & 474 & 8.60\e{-16} &2.03\e{-15}\\
    \addlinespace[.25em]
    &  $2\e{-4}$   & 7.43\e{5} & 1.00\e{3}& 741 &  1.05\e{-15} & 6.38\e{-16}\\
    \addlinespace[.25em]
     &  $2\e{-5}$   & 7.42\e{5} & 8.09\e{2} & 917 &  8.33\e{-16} & 5.83\e{-16}\\
    \addlinespace[.25em]
    \end{tabular}
    \caption{
    {\bf The speed-up of close evaluation of Newtonian potentials generated
    over a single mesh element $\Delta$}. 
    }
    \label{tab:near_speedup}
    \end{center}
\end{table}

\begin{figure}[h]
    \centering
    \begin{subfigure}{0.45\textwidth}
      \centering
      \includegraphics[width=\textwidth]{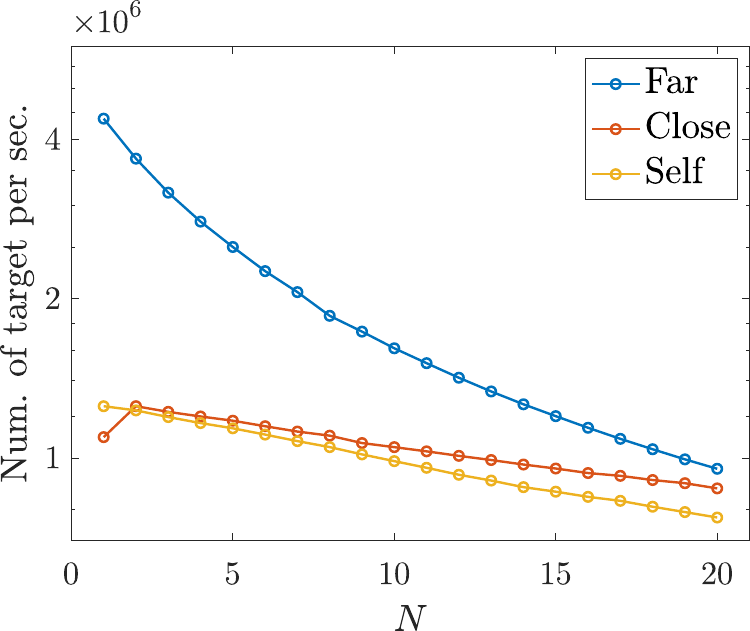}
      \caption{}
\label{fig:speed1_a}
    \end{subfigure}
    \quad
    \begin{subfigure}{0.45\textwidth}
      \centering
      \includegraphics[width=\textwidth]{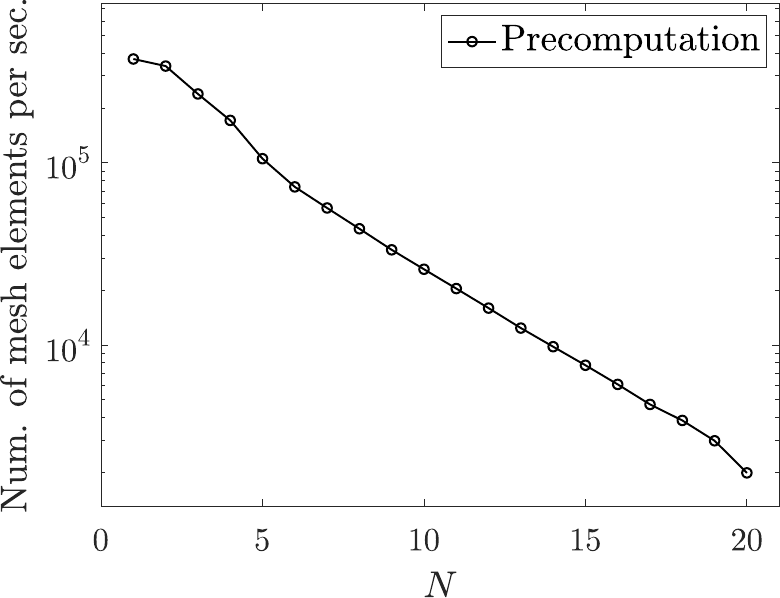}
      \caption{}
\label{fig:speed1_b}
    \end{subfigure}
  \caption{{\bf The speed of the far, close, and self-evaluations, and of
  the precomputations, over a single mesh element $\Delta$}. The $x$-axis
  denotes the order of approximation. The $y$-axes in (a) and (b) denote the
  number of targets the volume potential can be evaluated at and the number
  of mesh elements that the precomputation can be executed on, per second,
  respectively.   The label ``Far'' denotes the evaluation of the Newtonian
  potential via (\ref{for:green3}) by a Gauss-Legendre rule of order $N+2$
  over every edge of the mesh element. 
  The labels ``Close'' and ``Self'' denote the close and self-evaluation of
  the Newtonian potential via (\ref{for:green3}) by the Helsing-Ojala method
  of order $N+2$ over every edge of the mesh element. 
  }
  \label{fig:speed1}
\end{figure}

\begin{figure}[h]
    \centering
    \begin{subfigure}{0.48\textwidth}
      \centering
      \includegraphics[width=\textwidth]{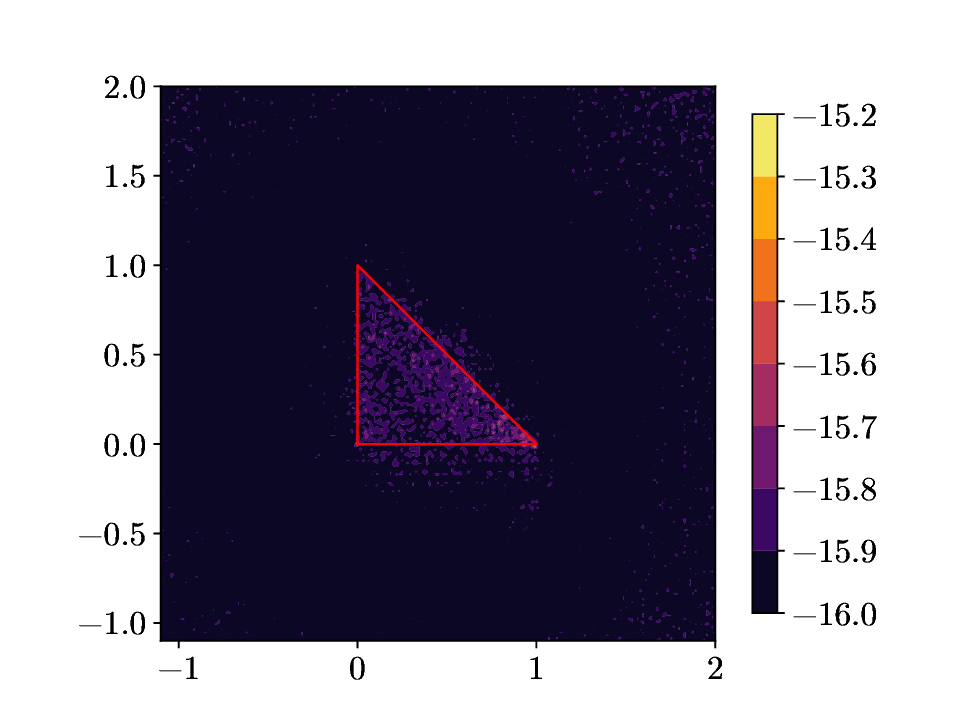}
      \caption{}
    \end{subfigure}
    \,\,\,
    \begin{subfigure}{0.48\textwidth}
      \centering
      \includegraphics[width=\textwidth]{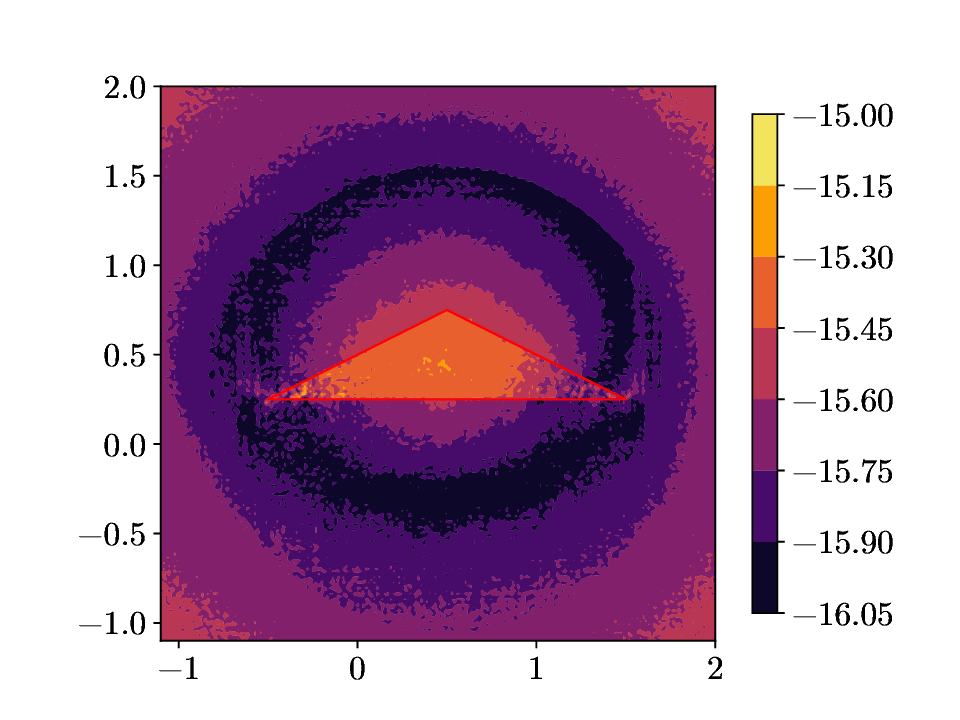}
      \caption{}
    \end{subfigure}

  \caption{{\bf Error contours}. We set the density function to be
  $f(x,y)=\sin(\frac{1}{2}xy+x+y)$, and the orders of the 2-D and 1-D
  monomial approximations to be 20 and 22, respectively.}
  \label{fig:pot_heat}
\end{figure}

\subsection{Poisson's equation}
In this section, we report the performance of our Newtonian potential
evaluation algorithm in the context of solving Poisson's equation 
  \begin{align}
\hspace*{-0em}\nabla^2 \varphi = &\,f \text{ in } \Omega, \notag\\
\hspace*{-0em}\varphi=&\,g \text{ on } \partial\Omega,
\label{for:poi}
  \end{align}
where 
\begin{align}
f(x,y)=9\cos(9x)\sin(6y)+16\cos\bigl(16y+\frac{8}{5}\bigr)-12\sin(12x),
\end{align}
and
\begin{align}
  \hspace*{-5.5em}g(x,y)=\varphi(x,y)|_{\partial\Omega}=&\,
  \Bigl(\frac{1}{12}\sin(12x)-\frac{1}{16}\cos\bigl(16y+\frac{8}{5}\bigr)
  -  \frac{1}{13}\cos(9x)\sin(6y)\Bigr)\Big|_{\partial\Omega},
\end{align}
for the domain $\Omega$ is displayed in Figure \ref{fig:pot_domain}.  The
boundary of this domain was created by the following procedure. First, we
fit a $C^\infty$ curve through a collection of points using the algorithm
described in~\cite{mohan}, and then up-sampled the curve to produce 
a new collection of data points. The final curve was constructed by
interpolating these new points using the interpolation formula and
$C^\infty$ interpolatory basis introduced in~\cite{zhang-ma}, where the
basis functions are translates of the product of the sinc function and the
Gaussian $\exp(-ax^2)$, with the parameter $a=0.1$. The final curve can be
recreated from this interpolation formula, together with the data points,
which we provide in~\url{https://doi.org/10.5281/zenodo.8401488}. Once the
boundary of $\Omega$ has been constructed, a mesh is created on $\Omega$
using the Gmsh mesh generator. To ensure that the boundary is well-resolved
by the triangle mesh, we post-process the resulting mesh
by performing additional subdivision of curved triangles on
high curvature regions.

It is easy to verify that the solution $\varphi$ to Poisson's equation
\eqref{for:poi} can be expressed as the sum of the Newtonian potential $u$
with density function $f$ over the domain $\Omega$, and the solution to the
Laplace equation
  \begin{align}
\nabla^2 u^h =&\, 0 \quad\text{in $\Omega$}, \notag\\
u^h=&\,g-u \quad \text{on $\partial\Omega$}. 
  \label{for:lap}
  \end{align}
In our implementation, we first compute $u$ using our algorithm, and then solve
the Laplace equation~(\ref{for:lap}) using the boundary integral equation
method~\cite{fds}.

In our first experiment, we aim to validate both the convergence rate and
the computational efficiency of our algorithm. To achieve this, we
discretize the domain $\Omega$ using a quasi-uniform triangular mesh. We
provide the problem sizes in Tables~\ref{tab:poi_prob_size22}
and~\ref{tab:poi_prob_size}.
Subsequently, we present a detailed performance analysis of our algorithm in
Table~\ref{tab:poi1}, and visualize its convergence rate in
Figure~\ref{fig:ord}. Our results demonstrate that our algorithm exhibits
the anticipated order of convergence. Additionally, when
the order is not extremely high (e.g., $\leq 14$), the time spent on
precomputation is small, and the computational costs associated
with near and self-interaction are comparable to those of the far field
interaction computation.

\begin{figure}[h]
    \centering
    \begin{subfigure}{0.45\textwidth}
      \centering
      \includegraphics[width=\textwidth]{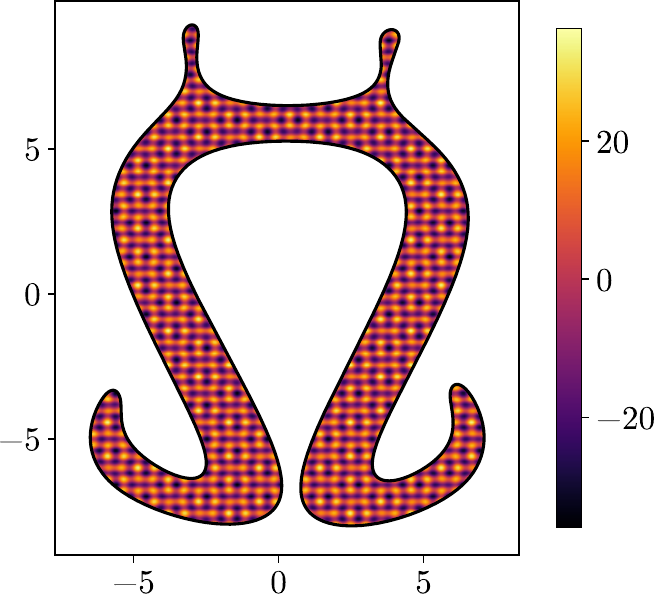}
      \caption{}
    \end{subfigure}
    \quad
    \begin{subfigure}{0.356\textwidth}
      \centering
      \includegraphics[width=\textwidth]{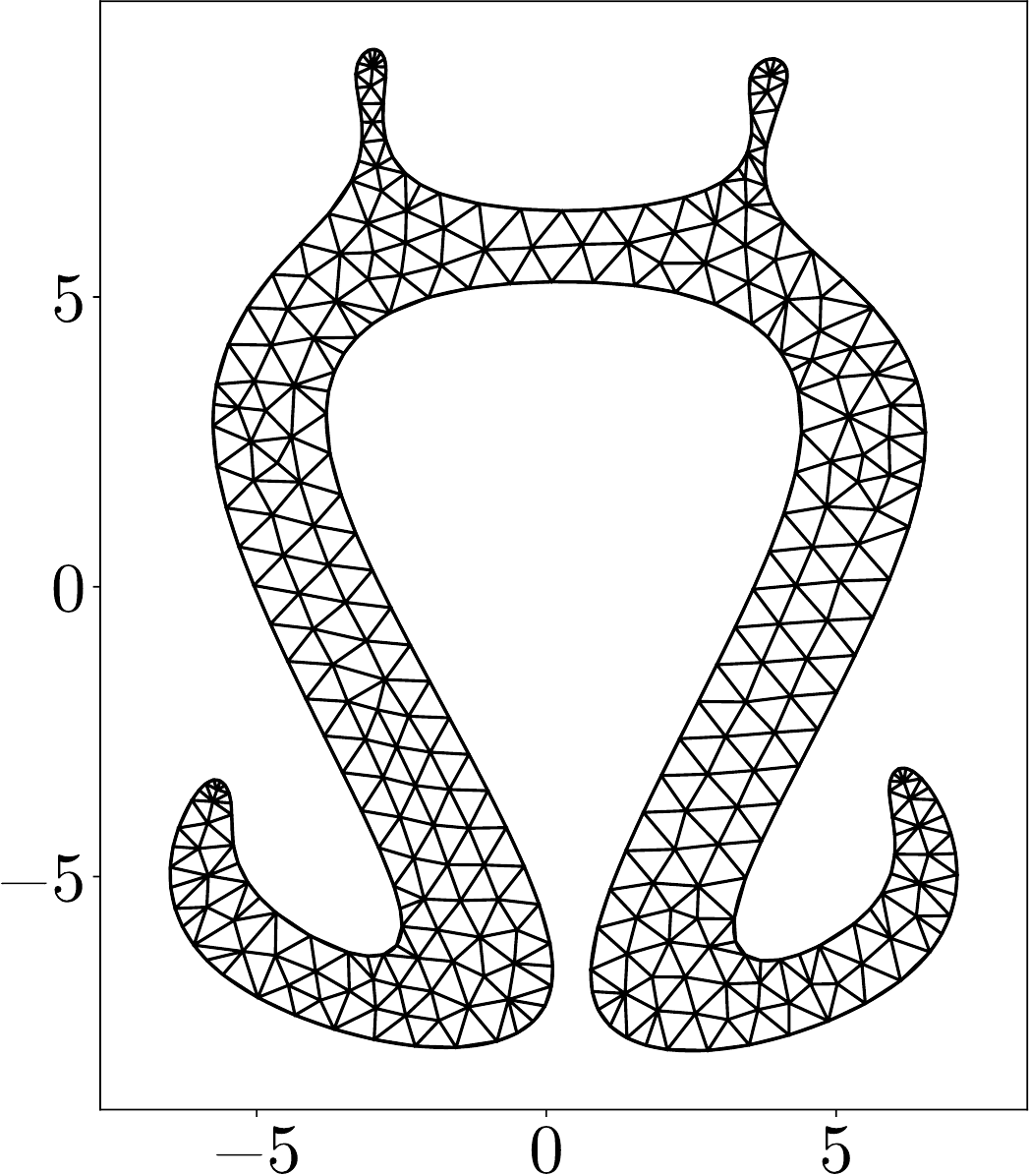}
      \caption{}
    \end{subfigure}
  \caption{{\bf The domain $\Omega$, the density function $f$, and an example mesh
  where~$h_0^{\max}=0.99$}.
  }
  \label{fig:pot_domain}
\end{figure}

\begin{table}[h!!]
    \begin{center}
    \begin{tabular}{cccc}
    $h_0^{\max}$ & $N_{\text{elem}}$ & $A^{\max}$ & $h_0^{\min}/h_0^{\max}$\\
    \midrule
    \addlinespace[.5em]
     0.99& 542 & 1.96 & 0.16\\
    \addlinespace[.25em]
      0.55 & 1844 & 4.20 & 0.17\\
    \addlinespace[.25em]
      0.28 & 6964 & 45.5 & 0.17\\
    \end{tabular}
    \caption{
    {\bf The total number of mesh elements, the maximum diameter
    (i.e., $h_0^{\max}$), aspect ratio (i.e., $A^{\max}$) of mesh
    elements, and the ratio between the largest and smallest
    triangles (i.e., $h_0^{\min}/h_0^{\max}$), for different meshes used in
    Poisson's equation experiment}. Note that the ratio between the minimum diameter and
    the maximum diameter of mesh elements is small, because the mesh size
    has to be smaller inside the two lobes (see Figure \ref{fig:pot_domain}).
    }
    \label{tab:poi_prob_size22}
    \end{center}
\end{table}

\begin{table}[h!!]
    \begin{center}
    \begin{tabular}{ccccc}
    $N_{\text{ord}}$&$h_0^{\max}$  &
    $N_{\text{tot}}^{\text{src}}$ & $N_{\text{tot}}^{\text{tgt}}$ & $\frac{N_{\text{tot}}^{\text{src}}}{N_{\text{tot}}^{\text{tgt}}}$  \\
    \midrule
    \addlinespace[.5em]
    8 & 0.99& 17886 & 24390 & \multirow{3}{*}{73.3\%}\\
    \addlinespace[.25em]
     & 0.55 &  60852 & 82980& \\
    \addlinespace[.25em]
     & 0.28 & 229812 & 313380 & \\
    \addlinespace[.5em]
    14 & 0.99& 27642 & 65040 & \multirow{3}{*}{42.5\%}\\
    \addlinespace[.25em]
     & 0.55 & 94044 &  221280 & \\
    \addlinespace[.25em]
     & 0.28 & 355164 & 835680 &\\
    \addlinespace[.5em]
    20 & 0.99& 37398 & 125202 &\multirow{3}{*}{29.9\%}\\
    \addlinespace[.25em]
     & 0.55 &  127236 &425964&\\
    \addlinespace[.25em]
     & 0.28 &480516& 1608684&\\
    \end{tabular}
    \caption{
    {\bf The total number of sources and targets for different orders of
    approximation $N_{\text{ord}}$ and for different mesh sizes $h_0$
    used in Poisson's equation experiment}.  We
    note that the total number of the mesh elements equals $552$, $1844$,
    $6898$ when $h_0=0.6$, $0.3$, $0.15$, respectively.  The sources are the
    Gauss-Legendre nodes of order $N_{\text{ord}}+2$ along the boundaries of
    the mesh elements, and the targets are the Vioreanu-Rokhlin nodes of
    order $N_{\text{ord}}$ over all mesh elements.  The values in the last
    column are equal to
    $6(N_\text{ord}+3)/\bigl((N_\text{ord}+1)\cdot(N_\text{ord}+2)\bigr)$,
    and show the reduction in the number of far field quadrature nodes, when
    the far field interactions are computed using Green's third identity.
    }
    \label{tab:poi_prob_size}
    \end{center}
\end{table}

\begin{table}[h!!]
    \begin{center}
    \begin{tabular}{ccccccccccc}
    $N_{\text{ord}}$ & $h_0^{\max}$ &$T_{\text{geom}}$&$T_{\text{init}}$& $T_{\ft}$& $T_{\nt}$ & $T_{\st}$ &
    $T_{\text{tot}}$ & $\frac{\#\text{tgt}}{\text{sec}}$& $E_{\text{poi}}$  \\
    \midrule
    \addlinespace[.5em]
    8 & $0.99$ &0.02 &  0.03  & 0.17 & 0.04 & 0.02& 0.29 &    8.53\e{4} &
    1.49\e{-3}\\
    \addlinespace[.25em]
     &  $0.55$  &0.06 & 0.07  & 0.60 & 0.15& 0.08& 0.97&
    8.57\e{4} & 1.01\e{-6}\\
    \addlinespace[.25em]
     &  $0.28$  & $0.24$ &0.26 & 1.72  & 0.67 & 0.32 & 3.21&
    9.77\e{4} & 1.90\e{-9}\\
        \addlinespace[.5em]
    14 & $0.99$  &0.05  &  0.09 & 0.34 & 0.15 & 0.08 &
    0.71 &9.13\e{4}&4.98\e{-7}\\
    \addlinespace[.25em]
     &  $0.55$ &0.18   & 0.31  & 1.05 & 0.81 & 0.36& 2.70 &
    8.19\e{4} &  9.05\e{-12}\\
    \addlinespace[.25em]
     &  $0.28$  &0.58  & 0.93  & 3.65 & 2.70 & 1.15& 9.01&
    9.27\e{4} & 3.75\e{-12}\\
    \addlinespace[.5em]
    20 & $0.99$ &0.20  &  0.96 & 0.99 & 0.33 & 0.18&2.67&
    4.69\e{4} & 7.71\e{-11}\\
    \addlinespace[.25em]
     &  $0.55$ & 0.31 & 2.84  & 2.34 & 1.09 & 0.62& 7.21&
    5.91\e{4} & 7.39\e{-12}\\
    \addlinespace[.25em]
     &  $0.28$ &1.26  & 10.5  & 7.20 & 4.46 & 2.55& 26.0&
    6.19\e{4} &  5.18\e{-12}\\
    \addlinespace[.25em]
    \end{tabular}
    \caption{
    {\bf Performance of our algorithm for solving Poisson's equation}. 
    The smallest value of $E_\text{poi}$ we report is $3.75 \times
    10^{-12}$.  We found that the loss of several digits accuracy was due to
    the condition number associated with the calculation of the curvature of
    $\partial \Omega$, which is represented as a global $C^\infty$ curve.
    Note that the meshes for different values of $N_{\text{ord}}$ with the
    same value of $h_0^{\max}$ are identical. 
    }
    \label{tab:poi1}
    \end{center}
\end{table}

\begin{figure}[h]
    \centering
    \begin{subfigure}{0.49\textwidth}
      \centering
      \includegraphics[width=\textwidth]{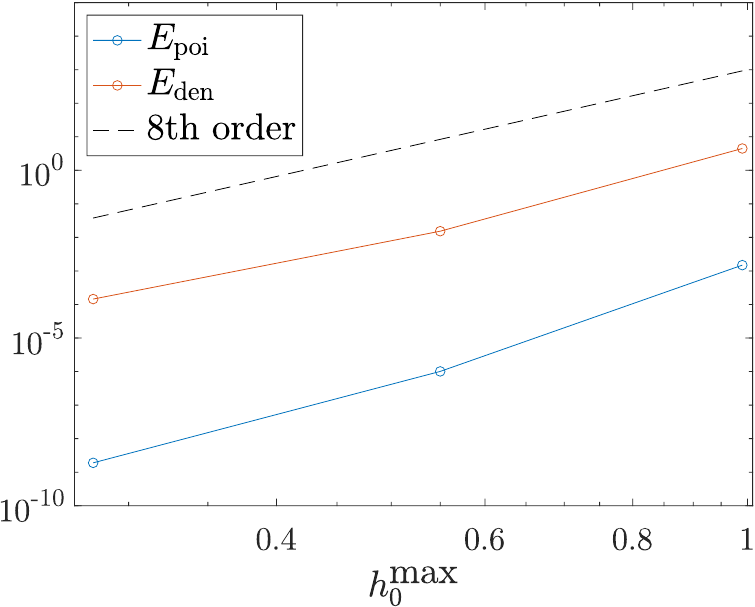}
      \caption{8th order}
    \end{subfigure}
    \begin{subfigure}{0.49\textwidth}
      \centering
      \includegraphics[width=\textwidth]{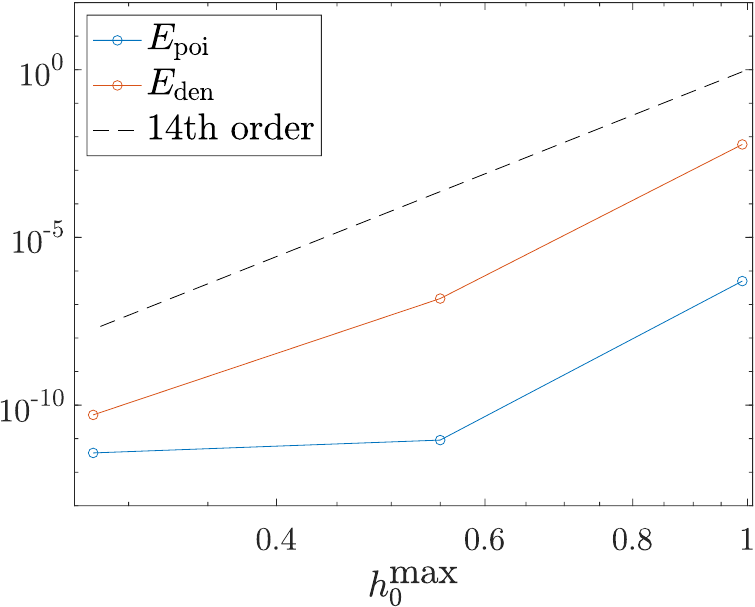}
      \caption{14th order}
    \end{subfigure}
  \caption{{\bf Order of convergence}.  We refer the readers to the caption
  of Table \ref{tab:poi1} for the reason why the error does not decay to
  machine precision in Figure (b). }
  \label{fig:ord}
\end{figure}

In the following example, we assess the performance and accuracy
of our algorithm
for an inhomogeneous density function and a domain
featuring multiscale features. Specifically, we define the density
function as the combination of a smooth function and three highly localized
spikes, given by:
  \begin{align}
\hspace*{-6em} f(x,y) = &\,2-32\sin(4(x+y)) +\notag\\
&\,2\sum_{i=1}^3 e^{-c_i^2
(x-a_i)^2-d_i^2(y-b_i)^2}\bigl(-c_i^2-d_i^2+2c_i^4(x-a_i)^2+2d_i^4(y-b_i)^2\bigr),
  \end{align}
where $(a_1,b_1,c_1,d_1):=(-3,6,10,10)$, $(a_2,b_2,c_2,d_2):=(2,-7,5,10)$
and $(a_3,b_3,c_3,d_3):=(3.8,8.6,15,20)$.
We set the boundary condition of Poisson's equation to be
\begin{align}
  \hspace*{-6.5em}g(x,y)=\varphi(x,y)|_{\partial\Omega}=&\,
  \sin(4(x+y))+(x^2-3y+8) + \sum_{i=1}^3 e^{-c_i^2
  (x-a_i)^2-d_i^2(y-b_i)^2}\Big|_{\partial\Omega}.
\end{align}
Additionally, we introduce a tiny lobe to the previously defined
domain $\Omega$. The density function and the domain are visualized in
Figure \ref{fig:pot_domain_fine}. We also set the order $N_\text{ord}$ of
our algorithm to be 14.

\begin{figure}[h]
    \centering
    \begin{subfigure}{0.99\textwidth}
      \centering
      \includegraphics[width=\textwidth]{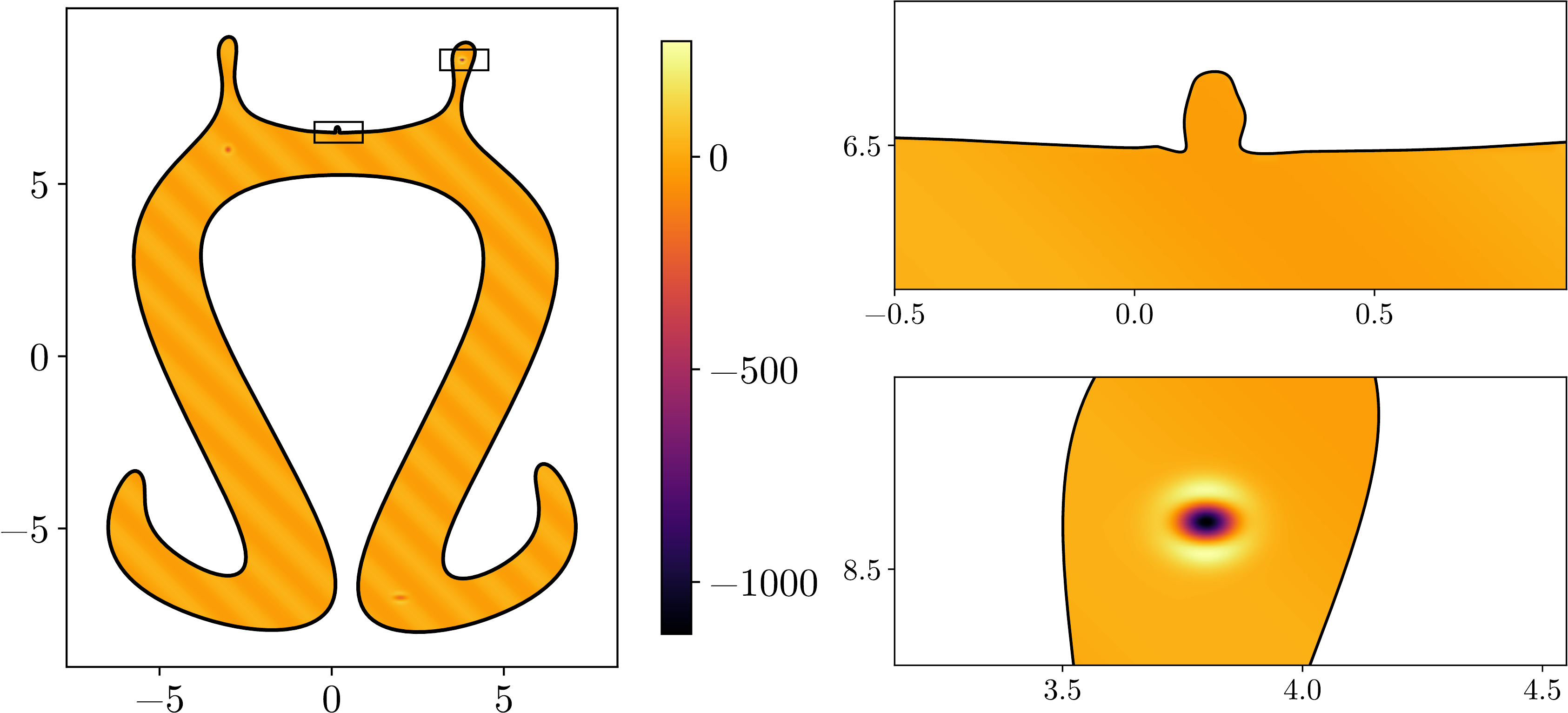}
    \end{subfigure}
  \caption{{\bf The domain $\Omega$ with a tiny lobe, and the inhomogeneous
  density function $f$}.
  }
  \label{fig:pot_domain_fine}
\end{figure}

To accurately represent the inhomogeneous density function, we refine the
initial mesh by 1-to-4 subdivisions until the density function over each
element can be approximated by
a polynomial of degree 14 within a specified error tolerance.
To resolve the tiny lobe, we use Gmsh to subdivide the lobe and its
neighborhood until we achieve the desired reduction in error.  We provide
the problem sizes and the performance of our algorithm in Table
\ref{tab:poi2}, and plot the corresponding meshes in Figure
\ref{fig:pot_domain_fine2}.  Note that the error tolerance required for
resolving the density is typically several orders of magnitude larger than
$E_{\text{tol}}$.

\begin{remark}
In our experiments, we found that the time spent on the far field
interaction computation performed by the FMM has a high variance, so we ran
the program several times and reported the experiment with the smallest FMM
computation time. We note that the runtimes of all other parts of our
algorithm have low variance.  
\end{remark}

\begin{figure}[h]
    \centering
    \begin{subfigure}{0.80\textwidth}
      \centering
      \includegraphics[width=\textwidth]{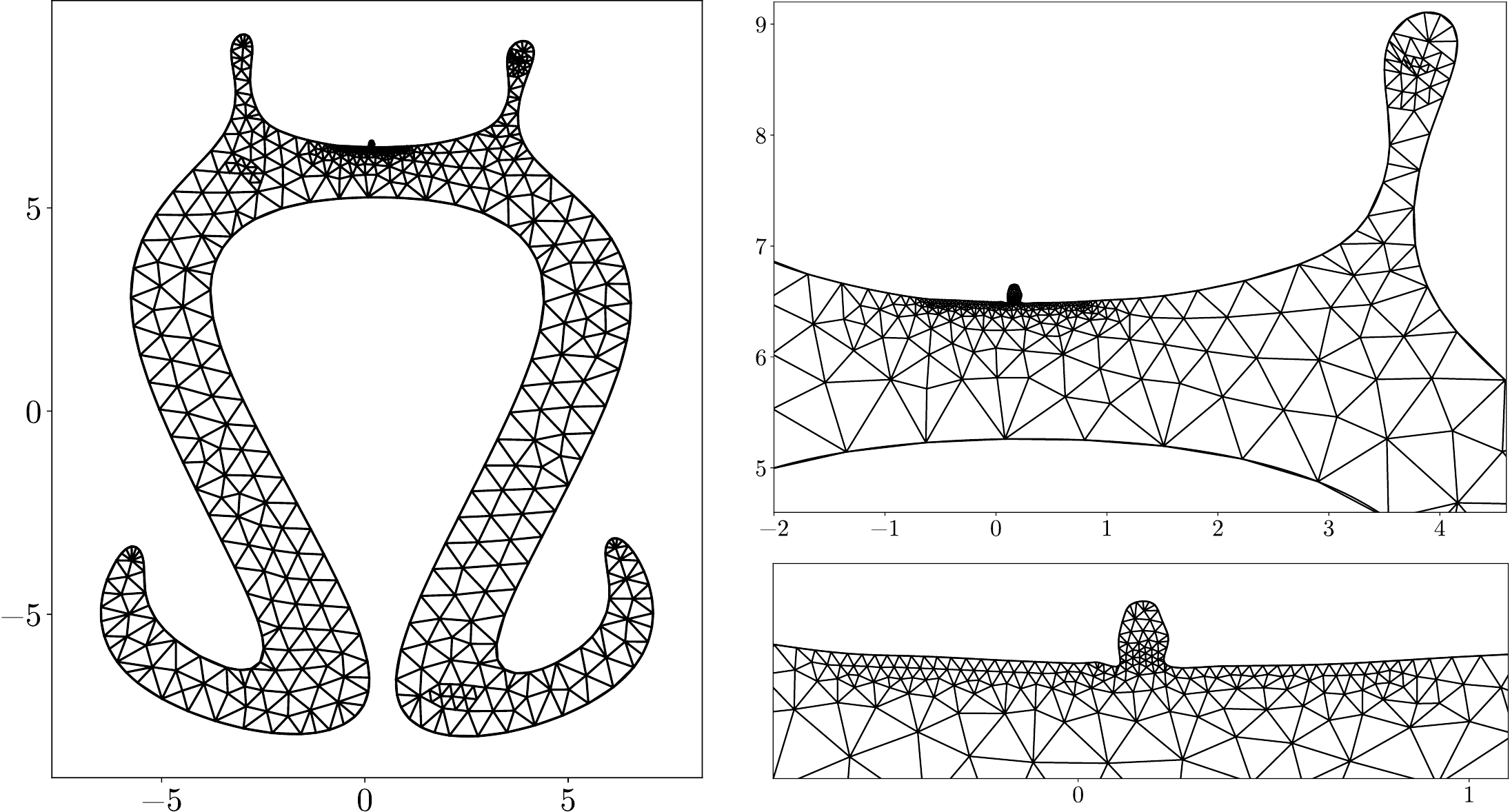}
      \caption{Error tolerance $= 10^{-5}$}
    \vspace*{0.95em}
    \end{subfigure}
    \begin{subfigure}{0.80\textwidth}
      \centering
      \includegraphics[width=\textwidth]{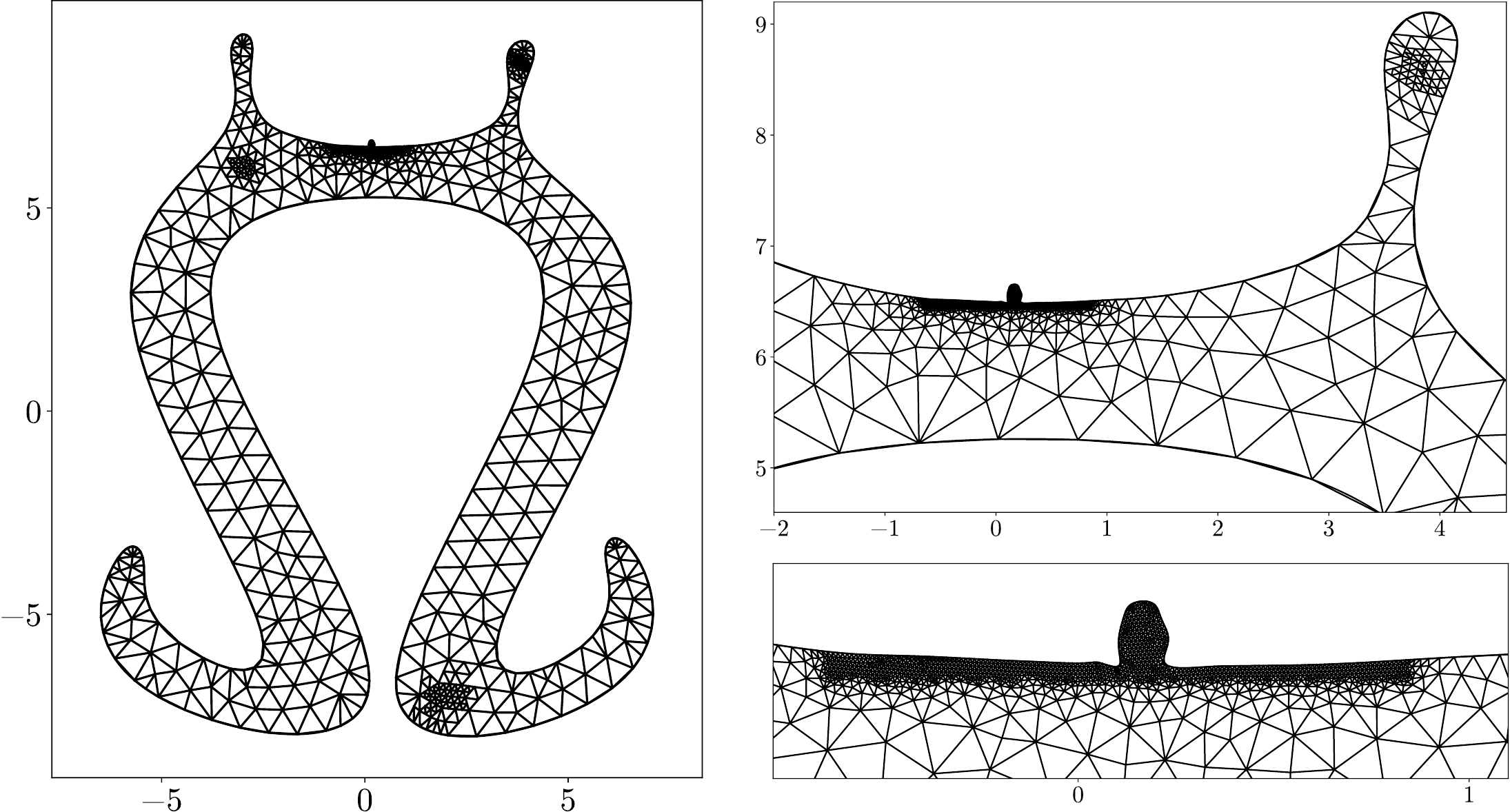}
      \caption{Error tolerance $= 10^{-10}$}
    \end{subfigure}
  \caption{{\bf The meshes used in the Poisson's equation experiments for 
  a multiscale domain and an inhomogeneous density function}.
  }
  \label{fig:pot_domain_fine2}
\end{figure}


\begin{table}[h!!]
    \begin{center}
    \begin{tabular}{ccccccccccc}
    $E_{\text{tol}}$ & $N_{\text{elem}}$ & $A^{\max}$ &
    $T_{\text{geom}}$&$T_{\text{init}}$& $T_{\ft}$& $T_{\nt}$ & $T_{\st}$ &
    $T_{\text{tot}}$ & $\frac{\#\text{tgt}}{\text{sec}}$ & $E_{\text{poi}}$  \\
    \midrule
    \addlinespace[.5em]
     $10^{-5}$& 1076 & 5.05 &0.15 &  0.19  & 0.60 & 0.19 & 0.15& 1.27 &    1.02\e{5} &
    5.75\e{-5}\\
    \addlinespace[.25em]
       $10^{-10}$ & 3258 & 2.41 &1.59 & 0.51  & 1.88 & 1.93& 0.72& 6.63&
    6.09\e{4} & 8.57\e{-10}\\
    \end{tabular}
    \caption{
    {\bf Performance of our algorithm for solving Poisson's equation when
    the domain has a multiscale feature and when the density function is
    inhomogeneous}. In this experiment, $N_{\text{ord}}=14$. We note that
    the time spent on the geometric algorithms (i.e., $T_{\text{geom}}$) is
    substantially higher than the one in the previous examples. This is
    because our geometric algorithm is not optimized for multiscale meshes.
    The near field interaction computation takes much longer for the same
    reason.
    }
    \label{tab:poi2}
    \end{center}
\end{table}

\section{Conclusions and further directions}
In this paper, we present a simple and efficient high-order algorithm for
the rapid evaluation of Newtonian potentials over a general planar domain.
Furthermore, we provide a justification for employing a monomial basis in
the context of high-order (up to order 20) bivariate polynomial
interpolation. This choice serves as a crucial component of our algorithm,
despite being commonly regarded as infeasible.

We note that our algorithm can be generalized to compute the Helmholtz (or
Yukawa) volume potential, provided that the anti-Helmholtizian (or
anti-Yukawaian) of a bivariate polynomial can be approximated accurately,
and that the Helmholtz (or Yukawa) layer potentials over the boundaries of
mesh elements can be efficiently computed to high accuracy (see, for example,
\cite{helsing2,fry0}).  Furthermore, the use of polynomial
interpolation in the monomial basis and Green's third identity 
generalize in a straightforward way to 3-D domains, but efficient algorithms
for the evaluation of surface Newtonian potentials are an area of active
research.

\section{Acknowledgements}
We are deeply grateful to Travis Askham, Shidong Jiang, Manas
Rachh, and the anonymous referees for their valuable advice and insightful
discussions.

\end{document}